\documentclass[12pt]{article}   	
\pdfoutput=1 

\usepackage{xcolor}

\usepackage{amsmath,amsfonts,amssymb,amsthm}
\usepackage{mathrsfs}

\usepackage{graphicx}
\usepackage{setspace}
\usepackage{float}
\usepackage{caption}
\usepackage{subcaption}
\RequirePackage{etex}
\usepackage{tabularx,multirow,hhline}
\usepackage{diagbox}
\usepackage{url}
\usepackage{hyperref}

\usepackage{tikz}
\usetikzlibrary{shapes.geometric}
\usetikzlibrary{shapes.misc}
\usetikzlibrary{calc}
\usepackage{tkz-graph}
\usepackage{rotating}
\usepackage{xcolor}
\usepackage{colortbl}
\definecolor{grey}{RGB}{211,211,211}

\usepackage{geometry}
\theoremstyle{definition} 
 
\newtheorem{prop}{Proposition}[section] 
 
\newtheorem{fact}{Fact}[section] 
\newtheorem{thm}{Theorem}[section] 

\newcommand{\cC}{{\cal C}}
\newcommand{\cH}{{\cal H}}
\newcommand{\cP}{{\cal P}}

\newcommand{\ccc}{c}

\newcommand{\model}[3]{\mathcal{M}(#1, #2, \lnot #3)}
\newcommand{\problem}[3]{\mathcal{P}(#1, #2, \lnot #3)}
\newcommand{\pair}[2]{(#1, \lnot #2)}
\newcommand{\chpair}{\pair{\cC}{\cH}}
\newcommand{\funct}[2]{f(#1,  \lnot #2)}
\newcommand{\genericmodel}{\model{n}{\cC}{\cH}}
\newcommand{\genericproblem}{\problem{n}{\cC}{\cH}}

\usepackage{appendix}
\usepackage{url}

\title{Proving hamiltonian properties in connected 4-regular graphs: an ILP-based approach}

\providecommand{\keywords}[1]
{
  \small	
  \textbf{Keywords---} #1
}

\usepackage{authblk}

\author[1]{Giuseppe Lancia}
\author[1,2]{Eleonora Pippia}
\author[1]{Franca Rinaldi}
\affil[1]{{\footnotesize Dipartimento di Scienze Matematiche, Informatiche e Fisiche - University of Udine, Via delle Scienze 206, 33100 UD, Italy}}
\affil[2]{{\footnotesize The Research Hub by Electrolux Professional S.p.A., Viale Treviso 15, 33170 PN, Italy}}

\date{}

\begin{document}

\maketitle
\begin{abstract}
In this paper we study some open questions related to the smallest order $f({\cal C},\lnot {\cal H})$ of a 4-regular graph which has a connectivity property ${{\cal C}}$ but does not have a hamiltonian property ${\cal H}$. In particular, ${\cal C}$ is either connectivity, 2-connectivity or 1-toughness and ${\cal H}$ is hamiltonicity, homogeneous traceability or traceability. A standard theoretical approach to these  questions had already been used in the literature, but in many cases did not succeed in determining the exact value of $f()$. Here we have chosen to use Integer Linear Programming and to 
 encode the graphs that we are looking for as the binary solutions to a suitable set of  linear inequalities. This way, there would exist a graph of order $n$ with certain properties if and only if the corresponding ILP had a feasible solution, which we have determined through a branch-and-cut procedure. By using our approach, we have been able to compute $f({\cal C},\lnot {\cal H})$ for all  the pairs of considered properties with the exception of ${\cal C}=$1-toughness, ${{\cal H}}=$traceability. Even in this last case, we have nonetheless significantly reduced the interval $[LB, UB]$ in which $f({\cal C},\lnot {\cal H})$ was known to lie. Finally, we have shown that for each  $n \geq f({\cal C},\lnot {\cal H})$ ($n \geq UB$ in the last case) there exists a 4-regular graph on $n$ vertices which has property ${\cal C}$ but not property ${\cal H}$.
\end{abstract}



\section{Introduction} \label{sec:introduction}
It is well-known that both the problems of deciding whether a graph is {\em traceable}, i.e., it admits a Hamilton path,  and  is {\em hamiltonian}, i.e., it admits a Hamilton cycle, are NP-complete  \cite{GJ79} even when the graph is $k$-regular with $k \geq 3$ \cite{P94}. Similar problems concern the fact that a graph contains a Hamilton path starting at each vertex or Hamilton paths between each pair of vertices. In the former case the graph is called {\em homogeneously traceable}, in the second case  {\em 
Hamilton-connected}. All the above problems are difficult also for regular graphs and this fact gave rise to a wide search for conditions that are  necessary and/or sufficient to guarantee a given hamiltonian property (see for instance \cite{B79,B78,G91, S84} for background and general surveys on the problems  and the papers \cite{BBV90, CS13, H86, J80} considering regular graphs). Clearly connectivity is a basic necessary condition for a graph satisfying each of the above properties; moreover, every hamiltonian graph must be 2-connected, i.e., at least two vertices have to be removed to disconnect the graph. A stronger necessary property for a graph to be hamiltonian was introduced by Chv\'atal \cite{C73} and is called $1$-toughness. Given $t \in \mathbb{R}$, $t > 0$, a graph $G = (V, E)$ is $t$-tough if for each set of vertices $S$ whose removal disconnects the graph the number of connected components of the graph induced by $V \setminus S$ is at most $\frac{ |S|}{t}$. As it is easy to verify, every hamiltonian graph is 1-tough, but the reverse statement does not hold in general. On the other hand, sufficient conditions are often based on the fact that the degree of the vertices of the graph is sufficiently high to guarantee a given hamiltonian property. In particular, when $k$-regular graphs are considered, no one of these properties may be guaranteed as the number of vertices increases. The issue of establishing the minimum order  of a $k$-regular graph that has a given connectivity property $\cC$, i.e., is either 1-tough   or 2-connected  or simply connected but does not satisfy a given hamiltonian property $\cH$, i.e., is not  Hamilton-connected  or is not hamiltonian  or is not homogeneously traceable  or is not traceable  has been considered in several papers.  In particular, relevant theoretical results determine lower bounds for these minimum orders. Two relevant conditions that guarantee that a graph is either Hamilton-connected or hamiltonian were stated by Ore in the following result.

\begin{thm}\label{thm:Ore1}(Ore \cite{O63}) Let $G$ be a graph with $n \geq 3$ vertices and let $d(v)$ denote the degree of vertex $v$. If for any pair of nonadjacent vertices $v$ and $w$ it holds that:\\
$~~~$i) $d(v)+d(w)\geq n+1$, then $G$ is Hamilton-connected;\\
$~~~$ii)    $d(v)+d(w)\geq n$  
then $G$ is hamiltonian.
\end{thm}
The above theorem implies, in particular, that every $k$-regular graph is Hamilton-connected if it has order $n \leq 2k-1$ and is hamiltonian if $3 \leq n \leq 2k$.

The  two stronger results concerning regular graphs are the following.
\begin{thm}\label{thm:CS} (Cranston and Suil \cite{CS13})
i) Every connected $k$-regular graph with at most $2k+2$ vertices is hamiltonian. Furthermore, all connected $k$-regular graphs on $2k+3$ vertices (when $k$ is even) and $2k+4$ vertices (when $k$ is odd) that are nonhamiltonian can be characterized.\\
ii) Every connected $k$-regular graph with at most $3k+3$ vertices has a Hamilton path. Furthermore, all connected $k$-regular graphs on $3k+4$ vertices (when $k\geq 6$ is even) and $3k+5$ vertices (when $k\geq5$ is odd) that have no Hamilton path can be characterized.
\end{thm}
\begin{thm}(Hilbig \cite{H86})\label{thm:hilbig} Let  $G$ be a $2$-connected, $k$-regular graph with at most $3k+3$ vertices. Then $G$ is hamiltonian or $G$ is the Petersen graph $P$ or $G$ is the $3$-regular graph obtained from $P$ by replacing one vertex with a triangle.
\end{thm}
The above theorems do not cover some issues concerning 4-regular graphs.  By Theorem \ref{thm:hilbig}, each 2-connected graph with $n \leq 15$ is hamiltonian and thus homogeneously traceable and traceable. On the other hand, as shown in \cite{BBV90}, there exists a 4-regular 1-tough graph with 18 vertices which is not hamiltonian. These facts leave open the following question: 
{\em Is it true that every 4-regular, 1-tough graph with at most 17 vertices is hamiltonian?} Bauer, Broersma and Veldman conjectured in \cite{BBV90} that this should be the case. We also observe that Theorem \ref{thm:CS} does not determine   the minimum order of a connected $k$-regular graph which is not traceable when $k=4$. These two facts seem to suggest that to answer the above questions (and similar open questions involving other pairs of connectivity conditions and hamiltonian properties) for 4-regular graphs one might need a different approach than a ``standard'' mathematical proof. 
In this paper, extending our preliminary work  \cite{LPR20}, we address the issue by adopting an Integer Linear Programming (ILP) approach.
We have proceeded as follows.

For each connectivity property 
$\mathcal{C}\in\{1$-toughness (1t), 2-connectivity (2c), connectivity (c)$\}$ and each hamiltonian property $\mathcal{H} \in \{$Hamilton-connectivity (HC), hamiltonicity (H), homogeneous traceability (HT), traceability (T)$\}$ we consider the problem $P(n, \cC, \lnot \cH)$ defined as

\vspace{2mm}
{\bf Problem $ P(n, \cC, \lnot \cH)$}: does there exist a 4-regular graph with $n$ vertices satisfying property $\cC$ and not satisfying property $\cH$?

\vspace{2mm} For each pair of properties $(\cC, \cH)$, we call any 4-regular graph having property $\cC$ but not property $\cH$ a $\chpair$-graph  and 
denote by  $\funct{\cal C}{\cal H}$ the minimum  $n$ for which problem $ P(n, \cC, \lnot \cH)$ has a positive answer, i.e., the minimum number of vertices in  a $(\cC, \lnot \cH)$-graph. For  each unknown value $\funct{\cal C}{\cal H}$  we have formulated problem $P(n, \cC, \lnot \cH)$ as an ILP problem  whose feasible solutions correspond to  the $\chpair$-graphs with $n$ vertices. Then we have solved the problem for increasing values of $n$ (chosen in a suitable range) so that $\funct{\cal C} {\cal H}$ was determined as the minimum $n$ for which the ILP model admits a feasible solution. Our computations allowed to almost complete Table \ref{tab:result1}, where the values $f(\cC, \lnot \cH)$ determined using our approach are written  in bold and the bounds previously known are written in normal font. In particular, we have shown that the question posed by Bauer, Broersma and Veldman has a positive answer and that every connected 4-regular graph with less than 18 vertices is traceable.   Furthermore, from our results it follows that for each considered pair of properties $\cC, \cH$ and for every $n \geq f\chpair$ there exists a $\chpair$-graph with $n$ vertices. The only value that remains undetermined  is $f(1t, \lnot T)$, i.e., the minimum order of a 4-regular 1-tough graph which is not traceable. However, even for this case, we were able to restrict the range to which this value belongs. We remark that the three values $f(\cC, \lnot HC)$, that we report in Table \ref{tab:result1} for sake of  completeness, were already known.

\begin{table}[t]

    \centering
    {\footnotesize
    \begin{spacing}{1.5}
\begin{tabular}{c||c|c|c|c}
& Hamilton- 	& hamiltonicity 		& homogeneous  & traceability \\
& connectivity & & traceability  & \\
\hhline{==|=|=|=}
connectivity & $f(c, \lnot HC) = 8$ 		& $f(c, \lnot H) = 11$ 			 & $f(c, \lnot HT) = 11$ 			& $f(c, \lnot T) \geq 16$  \ \\
		& 		&  & & $\bf{ f(c, \lnot T) = 18}$	\\
\hline
$2$-connectivity &  $f(2c, \lnot HC) = 8$ 		&  $f(2c, \lnot H) \geq 16$ 		 & $f(2c, \lnot HT) \geq 16$			& $f(2c, \lnot T) \geq 16$ \\\
		& 	& $\bf{ f(2c, \lnot H) = 16}$ & $\bf{ f(2c, \lnot HT) = 16}$ & $\bf{ f(2c, \lnot T) = 22}$	\\
\hline
$1$-toughness & $f(1t, \lnot HC) = 8$ 		& $ 16 \leq f(1t, \lnot H) \leq 18 $			 & $ f(1t, \lnot HT) \geq 16$			& $ f(1t, \lnot T) \geq 16$ \\
		&  & $\bf{f(1t, \lnot H) =18}$ & $\bf{f(1t, \lnot HT) = 20}$ &  $\bf{22 \leq f(1t, \lnot T) \leq 40}$	\\
\end{tabular}
\end{spacing}
}
\caption{Known and new bounds for the minimum order $f(\cC, \lnot \cH)$ of a 4-regular graph that satisfies property $\cC$ but not property $\cH$. The new bounds appear in 
bold.} \label{tab:result1}
\end{table}

As it is well known, a feasibility ILP  problem consists in finding an integer solution to a finite set of linear inequalities. Since the number of inequalities required to model every problem $ P(n, \cC, \lnot \cH)$ happens to be exponential in $n$,  all the models were solved by using a branch and cut procedure. Despite the values  of $n$ used in the computations are relatively small (always less than 22), the dimension and the structure of the ILP models make their straightforward solution impossible in a reasonable time. For this reason we have adopted two fundamental strategies to reduce the computation times: a preliminary analysis that allowed us to conveniently split each model in few  subproblems in which some variables may be fixed  and the use of a symmetry-breaking technique called {\em orbital branching} \cite{OLRS11} to reduce the symmetry of the subproblems. 

The use of ILP as a technique to design a combinatorial 
object with given properties (such as, for instance, a counterexample to some hypothesis that one might have formulated)
is not new, but is not as popular as it should probably be. For instance, Pulaj et al. applied   ILP to  study the size of counterexamples to the {\em union-closed set conjecture} 
(\cite{P17, PRT20}), while 
 Caprara et al.\cite{Cap15} used ILP to find 
counterexamples to a property that fractional bin packing solutions should satisfy when rounded up to integer. Finally,
in \cite{TSSW96}, Trevisan et al. used ILP to build ``gadgets'' that
can turn a combinatorial problem into another. Through these gadgets, the authors were able to construct instances which they used to
improve the approximability/inapproximability factors of some important combinatorial optimization problems.

The remainder of the paper is organized as follows. In Section \ref{sec:notation} we introduce the notation and recall some known results. In Section \ref{sec:preliminary} we describe a preliminary analysis about 4-regular, 2-connected graphs that are not 1-tough. Also this analysis is done using an ILP method. In Section \ref{sec:model} we present our ILP models, the  branch-and-cut procedure to solve them and some strategies required to obtain an effective procedure. The obtained results are described  in Section \ref{sec:results}.  Section \ref{sec:computationalresults} is devoted to discuss the main implementation issues and the computational experiments. Finally,  we draw some conclusions in Section \ref{sec:conclusions}. 

\section{Notation and known results}
\label{sec:notation}

Let $G = (V, E)$ be an undirected graph. The graph is called {\em $k$-regular} if every vertex has degree $k$.
For each $S \subseteq V$  we denote by $\partial(S)$ the set of edges of $G$ having an endpoint in $S$ and the other in $V \setminus S$. Moreover, we denote by $G[S]$ the subgraph of $G$ induced by $S$, i.e., the graph with vertex set  $S$ and edge set $E(S)$, the set containing all the edges of $E$ with both endpoints in $S$.

The graph $G$ is called {\em connected} if it contains a path between each pair  of vertices, and is called {\em 2-connected} if the graph $G[V \setminus \{i\}]$ is connected for each vertex $i \in V$. Let $\ccc(G)$ denote the number of connected components of $G$. The graph $G$ is called {$t$-tough}, $t \in \mathbb{R}_{+}$, if  for every subset $S\subseteq V$ with $\ccc(G[V \setminus S]) > 1$ it is $|S| \ge t\, \ccc(G[V \setminus S])$. In particular, $G$ is 1-tough if one cannot create $c$ components by removing less than $c$ vertices. Clearly, every 1-tough graph is 2-connected and thus connected. We remark that  the problem of deciding if a graph is $t$-tough is NP-hard even for $t=1$ \cite{BHS90} and  for regular graphs  \cite{BHMS97}. For an excellent survey on toughness in graphs the reader is referred to   the paper by Bauer, Broersma, and Schmeichel \cite{BBS06}.

A Hamilton cycle (or path) of $G$ is a cycle (respectively, a path)  that visits each vertex of $V$ exactly once.  A graph $G$ is called {\em traceable}  if it contains a Hamilton path and is called {\em hamiltonian}  if it contains a Hamilton  cycle. Moreover, $G$ is called {\em homogeneously traceable} if for each vertex $i \in V$ it contains a Hamilton path beginning at $i$ and is called
 {\em Hamilton-connected} if for each pair of vertices $i, j \in V$ it contains a Hamilton path starting at $i$ and ending in $j$. The next claims collect some properties that immediately follow from the above definitions. 
 
 \begin{fact}\label{prop:implications1}
Let $G$ be a graph with at least three vertices. Then: 
if $G$ is Hamilton-connected then $G$ is hamiltonian, if $G$ is hamiltonian then $G$ is homogeneously traceable and if $G$ is homogeneously traceable then $G$ is traceable.

\end{fact}
\begin{fact}\label{prop:implications2}
Every homogeneously traceable graph with at least three vertices  is 1-tough.
\end{fact}
\begin{proof} Assume that the graph $G = (V, E)$ is homogeneously traceable and, given a nonempty $S \subset V$, let $P$ be any Hamilton path beginning at a vertex of $S$. Since by removing from $P$ the vertices of $S$ one obtains at most $|S|$ subpaths of $P$ and each node of $V \setminus S$ lies on exactly one of these subpaths, the graph $G[V\setminus S]$ has at most $|S|$ connected components. 
\end{proof}

Let us now consider the hamiltonian properties of the $4$-regular graphs and, in particular, what is already known about the minimum order $f(\cC, \lnot \cH)$ of a 4-regular graph that satisfies property $\cC \in \{$connectivity (c), 2-connectivity (2c), 1-toughness (1t)$\}$ and does not satisfy the property $\cH \in \{$Hamilton-connectivity (HC), hamiltonicity (H), homogeneous traceability (HT), traceability (T)$\}$.  As far as  property $HC$ is concerned,  it is easy to verify that the complete bipartite graph $K_{4,4}$ is 1-tough and does not contain any Hamilton path connecting  two nonadjacent vertices, thus $f(1t, \lnot HC) \leq 8$. On the other hand, by Theorem \ref{thm:Ore1} i),    $f(\cC, \lnot HC) > 7$ for any considered connectivity property $\cC$. This implies  $f(\cC, \lnot HC) = 8$ for any  property $\cC$. With regards to the other hamiltonian properties, we observe that any 2-connected 4-regular graph of order at most 15 is hamiltonian by Theorem \ref{thm:hilbig} and, by Fact \ref{prop:implications1}, this implies $f(\cC, \lnot \cH) \geq 16$ for each $\cC \in \{2c, 1t\}$ and $\cH \in \{H, HT, T\}$. Moreover, $f(c, \lnot H) \geq 11$ and $f(c, \lnot T) \geq 16$ by Theorem \ref{thm:CS}. Since the nonhamiltonian graph with 11 vertices in Fig.~\ref{Figure:n11} (reported in \cite{CS13}) does not contain any Hamilton path starting at vertex $v$,  we may conclude that $f(c, \lnot HT) = f(c, \lnot H) = 11$. Finally, the upper bound 18 for $f(1t, \lnot H)$ is due to the 1-tough but not hamiltonian graph  in Fig.~\ref{Figure:n18} which has been proposed in \cite{BBV90}.

Thanks to the above remarks we can fill Table \ref{tab:result1} with the known lower and upper bounds on the values $f(\cC, \lnot \cH)$ (reported not in bold).

\begin{figure}[ht] 
\centering 

\begin{minipage}[t]{0.45\linewidth}
\centering

\begin{tikzpicture}[scale=0.7]
 create the vertex
\foreach \x in {-1, 0.5, 3.5, 5}{
\draw[fill] (\x,0) circle (1.5pt);
\draw[fill] (\x,2.4) circle (1.5pt);
};

\draw[fill] (-2.5,1.2) circle (1.5pt);
\draw[fill] (6.5,1.2) circle (1.5pt);

\node[shape=circle, radius=2pt] (v) at (2,1.2){$v$} ;

\draw(0.5,0) -- (-1,0);
\draw (0.5,0) -- (-1,2.4);
\draw (0.5,0) -- (-2.5,1.2);
\draw (0.5,2.4) -- (-1,0);
\draw (0.5,2.4) -- (-1,2.4);
\draw (0.5,2.4) -- (-2.5,1.2);
\draw (-1,0) -- (-1,2.4);
\draw (-1,0) -- (-2.5,1.2);
\draw (-1,2.4) -- (-2.5,1.2);

\draw (0.5,2.4) -- (v);
\draw (0.5,0) -- (v);
\draw (3.5,0) -- (v);
\draw (3.5,2.4) -- (v);
\draw (3.5,0) -- (5,2.4);
\draw (3.5,0) -- (5,0);
\draw(3.5,0) -- (6.5,1.2);
\draw (3.5,2.4) -- (5,2.4);
\draw (3.5,2.4) -- (5,0);
\draw(3.5,2.4) -- (6.5,1.2);
\draw (5,0) -- (5,2.4);
\draw (5,2.4) -- (6.5,1.2);
\draw (5,0) -- (6.5,1.2);
\draw (5,2.4) -- (6.5,1.2);

\node (v) at (0,-1.5){};
%



%

%

\end{tikzpicture}
\caption{Connected non-homogeneously traceable graph with  $n = 11$.} \label{Figure:n11}
\end{minipage}
\hspace{0.5cm}
\begin{minipage}[t]{0.45\linewidth}
\centering
\begin{tikzpicture}[scale=0.9]

\draw (0,1) -- (1,1);

\foreach \x in {0,1,2,3,4,5}{
\draw[fill] (\x,1) circle (1.5pt);
\draw[fill] (\x,2) circle (1.5pt);};

\foreach \x in {0.5, 2.5, 4.5}{
\draw[fill] (\x,0) circle (1.5pt);
\draw[fill] (\x,3) circle (1.5pt);};

\foreach \x in {0, 2, 4}{
\draw (\x, 1) -- (\x+1, 1);
\draw (\x, 2) -- (\x+1, 2);
\draw (\x, 1) -- (\x, 2);
\draw (\x+1, 1) -- (\x+1, 2);
\draw (\x, 1) -- (\x+1, 2);
\draw (\x, 2) -- (\x+1, 1);};

\foreach \x in {0.5, 2.5, 4.5}{
\draw (\x, 0) -- (\x-0.5,1);
\draw (\x, 0) -- (\x+0.5,1);
\draw (\x, 3) -- (\x-0.5,2);
\draw (\x, 3) -- (\x+0.5,2);};

\draw (0.5,0) -- (2.5,0) -- (4.5,0);
\draw (0.5,3) -- (2.5,3) -- (4.5,3);
\path[-](0.5,0) edge [bend right=30] (4.5,0);
\path[-](0.5,3) edge [bend left=30] (4.5,3);
\end{tikzpicture}
\caption{ 1-tough nonhamiltonian graph with $n=18$.}
\label{Figure:n18}
\end{minipage}

\end{figure}

\section{Replacing 1-thougness by 2-connectivity: a preliminary analysis}\label{sec:preliminary}
Our strategy is based on the use of ILP to model each problem $P(n, \cC, \lnot \cH)$, i.e., the problem to find if there exists a 4-regular graph satisfying property $\cC$ but not property $\cH$. The variables of the model represent the edges of the sought graph. 
While it is easy to state a set of constraints which imply that a graph is 2-connected (or just connected), dealing with the constraints which enforce a graph to be 1-tough is not a simple task. Indeed, to determine if a graph is 1-tough is NP-complete \cite{BHS90}. Since every 1-tough graph is 2-connected,  even when solving problems  $P(n, 1t, \lnot \cH)$ it is then convenient to solve the relaxed problem $P(n, 2c, \lnot \cH)$. If the search fails, one can conclude that $f(1t, \lnot \cH)$, as well as  $f( 2c, \lnot \cH)$, is larger than $n$. Otherwise, if the model succeeds and finds a 2-connected graph which is not 1-tough, one should add suitable constraints to make this graph infeasible and continue the search for a 1-tough graph. As a final remark we observe that, since by Theorem \ref{thm:hilbig} every 2-connected 4-regular graph with $n \leq 15$ is 1-tough, one may expect that for slightly larger values of $n$  the 2-connected not 1-tough graphs are quite few and may be characterized. The knowledge of these graphs will be usefully exploited to reduce the computational effort required to solve problem $P(n, 1t, \lnot \cH)$, specially for $n =16, 17$.   

This preliminary analysis has the objective to study if it is possible to have 4-regular, 2-connected graphs with $n$ vertices which are not 1-tough and, in this case, to characterize their structure (in the following we call a graph with these properties a $\varphi$-graph). Also this analysis has been carried out  using an ILP approach. Given $n$ and  $k = 2, \dots, \lfloor\frac{n}{2}\rfloor$, let us denote by $v(n, k)$ the maximum number of connected components that can result by removing $k$ vertices from a 4-regular 2-connected graph with $n$ vertices. Clearly, if $v(n, k) \leq k$ for every $k$, then every 2-connected graph with $n$ vertices is 1-tough.  Otherwise, at least for $n = 16, 17$, the structure of the $\varphi$-graphs can be easily determined.

Before presenting the ILP model, let us outline some simple properties.
Given a 4-regular 2-connected graph $G = (V, E)$ and a subset $S \subseteq V$ with  $|S| = k$, let $W_{1}, \dots, W_{t}$ be the vertex-sets of the $t$ connected components of the graph $G[V \setminus S]$ and  $n_{r} := |W_{r}|$, for $r = 1, \dots, t$.

\begin{prop} \label{prop:1} 
For each $r = 1, \dots, t$ it is $n_{r} \geq 5-k$.
\end{prop}
\begin{proof} Assume $n_{r} \leq 4 - k$ for some $r$. 
Since each vertex $v$ of $W_{r}$ has degree 4 and can be adjacent to at most $n_{r} - 1 \leq 3 - k$ vertices of $W_{r}$, $v$ must be adjacent to at least $k+1$ vertices in $S$, a contradiction. 
\end{proof}

\begin{prop}  \label{prop:2} For each $r = 1, \dots, t$ it is $|\partial(W_{r})| \geq m_{r}$ with $m_r := \max \{2, n_{r} (5 - n_{r})\}$. This in particular implies $\sum_{r=1}^{t} m_{r} \leq 4 k$.
\end{prop}

\begin{proof} The 2-connectivity of $G$ implies $|\partial(W_{r})| \geq 2$. If $n_r\leq 4$, each vertex of $W_{r}$ must be adjacent to at least $5 - n_{r}$ vertices of $S$, so $|\partial(W_{r})| \geq m_{r}$.  Since $\sum_{r=1}^{t} |\partial(W_{r})| = |\partial(S)|  \leq 4k$ the second statement holds. 
\end{proof}

For each $n$ and $k$ we can compute an upper bound $v'(n, k)$ to the value $v(n, k)$ by solving the following ILP problem. Let $x_{i}$  be an integer variable representing the number of components of cardinality $i$ in the graph $G[V\setminus S]$ and  $m_i:=\max\{2,5i-i^2\}$. By Proposition \ref{prop:1} we can assume that $i$ goes from $s(k):=\max\{1,5-k\}$ to $n - k$. Let us consider the model $\mathcal{Q}_{n,k}$:

\begin{eqnarray}
v'(n,k) := \max\sum_{i=s(k)}^{n-k} x_i \\
 \label{vincolo1} \sum_{i=s(k)}^{n-k} i\;x_i &=& n-k \\
\label{vincolo2} \sum_{i=s(k)}^{n-k} m_i x_i &\leq& 4k  \\
 x_i&\in&\mathbb{N}\quad \quad \quad \forall~i = s(k), \dots, n-k.
\end{eqnarray}

The objective function counts the number of components of the graph $G[V\setminus S]$, the constraints (\ref{vincolo1}) state that the total number of vertices in these components must be $n-k$ and the constraints (\ref{vincolo2}) require that the property stated in Proposition \ref{prop:2} is satisfied. If there exists a 2-connected graph with $n$ vertices which is not 1-tough,  then it must be $v'(n, k) > k$ for some $k$. 

By Theorem \ref{thm:hilbig} every $2$-connected, $4$-regular graph with $n\leq 15$ is $1$-tough. So, in order to close the conjecture by Bauer, Broersma and Veldman,  we first focused on the cases $n=16$ and $n=17$.

By solving problem $\mathcal{Q}_{n,k}$ for $n = 16$, it turns out that $v'(16,k) > k$ only for $k=2$, in which case it is  $v'(16, 2) = 3$. 
 The optimal solution is $x^*_4=1$, $x^*_5=2$, $x^*_i=0$ for $i\not=4,5$.
 It is easy to verify that there is just one $\varphi$-graph compatible with this solution, namely the graph in Figure \ref{Figure:n16}. Note that if we remove the vertices in  $S = \{v_{1}, v_{2}\}$ we obtain a graph with 3 components, one with 4 vertices and two with 5 vertices. 
By solving again the problem  $\mathcal{Q}_{16,2}$ with the additional constraint $x_3\geq 1$ or  the problem  $\mathcal{Q}_{16,2}$ with the additional constraint $\sum_{i\geq 6} x_i\geq 1$, we obtain optimal value 2. This means that $x^{*}$ is the unique solution of $\mathcal{Q}_{16,2}$. 

 \begin{figure}[t]  
 \vspace{-0.5cm}
 	\centering
 	\begin{tikzpicture}[scale=0.40, transform shape]
 	\tikzset{cross/.style={cross out, draw=black, minimum size=2*(#1-\pgflinewidth), inner sep=0pt, outer sep=0pt},
 		cross/.default={1pt}}
 	
 	\draw[xshift=0.0] (1.618,1.17557) \foreach \x in {108,180,252,324} {
 		-- (\x:2)};
 	\foreach \x in {252,324,36,108,180}
 	\draw[fill] (\x:2) circle (2pt);
 	
 	\draw[xshift=0.0] (-2,0) \foreach \x in {324,108,252,36,180} {
 		-- (\x:2)};

 	\draw[xshift=8cm] (-1.618,-1.17557) \foreach \x in {288,360,72,144} {
 		-- (\x:2)};
 	[xshitf=8cm]\foreach \x in {72,144,216,288,360}
 	\draw[fill, xshift=8cm] (\x:2) circle (2pt);
 	
 	\draw[xshift=8cm] (2,0) \foreach \x in {144,288,72,216,360} {
 		-- (\x:2)};
 	
 	\draw[xshift=4cm] (1.24,1.24) \foreach \x in {45, 135,225,315,45} {
 		-- (\x:1.7557)};
 	[xshitf=4cm]\foreach \x in {45, 135,225,315}
 	\draw[fill, xshift=4cm] (\x:1.7557) circle (2pt);
 	
 	\draw[xshift=4cm] (1.24,1.24) -- (225:1.7557);
 	\draw[xshift=4cm] (1.24,-1.24) -- (135:1.7557);

 	\node[shape=star , star points=4, inner sep=0pt, draw, minimum size=6pt, xshift=4cm, thick] (v1) at (0,3) {};
 	\node[shape=circle, xshift=4cm] (1) at (0,3.4) {\LARGE$v_1$} ;
 	\draw[xshift=4cm] (0,3) -- (45:1.7557);
 	\draw[xshift=4cm] (0,3) -- (135:1.7557);
 	
 	\node[shape=star , star points=4, inner sep=0pt, draw, minimum size=6pt, xshift=4cm, thick] (v1) at (0,-3) {};
 	\node[shape=circle, xshift=4cm] (2) at (0,-3.4) {\LARGE$v_2$} ;
 	\draw[xshift=4cm] (0,-3) -- (-45:1.7557);
 	\draw[xshift=4cm] (0,-3) -- (-135:1.7557);
 	
 	\draw[xshift=4cm] (0,-3) -- (-4.0+1.618,-3.0+1.8);
 	\draw[xshift=4cm] (0,-3) -- (4.0-1.618,-3.0+1.8);
 	
 	\draw[xshift=4cm] (0,3) -- (-4.0+1.618,3.0-1.8);
 	\draw[xshift=4cm] (0,3) -- (4.0-1.618,3.0-1.8);
 	
 	\end{tikzpicture}
 \vspace{-0.2cm}
 	\caption{The unique 4-regular 2-connected graph with $n=16$ that is not $1-$tough.}\label{Figure:n16}
 \end{figure}
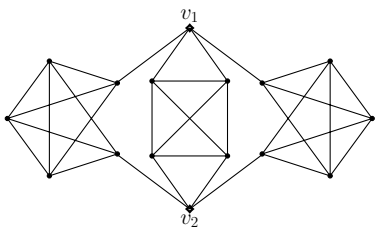

\begin{figure}[htbp]
  \centering
\begin{subfigure}[b]{0.4\linewidth}
\begin{tikzpicture}[scale=0.4, transform shape]

\draw[xshift=0.0] (1.618,1.17557) \foreach \x in {108,180,252,324} {
-- (\x:2)};
\foreach \x in {252,324,36,108,180}
\draw[fill] (\x:2) circle (2pt);

\draw[xshift=0.0] (-2,0) \foreach \x in {324,108,252,36,180} {
-- (\x:2)};

\draw[xshift=8cm] (-1.618,-1.17557) \foreach \x in {288,360,72,144} {
-- (\x:2)};
[xshitf=8cm]\foreach \x in {72,144,216,288,360}
\draw[fill, xshift=8cm] (\x:2) circle (2pt);

\draw[xshift=8cm] (2,0) \foreach \x in {144,288,72,216,360} {
-- (\x:2)};

\draw[xshift=4cm, yshift=4cm] (1.17557,-1.618) \foreach \x in {-54, 18, 90, 162, 234} {
-- (\x:2)};
[xshitf=4cm, yshift=4cm]\foreach \x in {-54, 18, 90, 162, 234}
\draw[fill, xshift=4cm, yshift=4cm] (\x:2) circle (2pt);

\draw[xshift=4cm, yshift=4cm] (1.17557,-1.618) \foreach \x in {90,234,18,162,-54} {
-- (\x:2)};

\node[shape=star , star points=4, inner sep=0pt, draw, minimum size=6pt, xshift=4cm, thick] (v1) at (45:1.7557) {};
\node[shape=star , star points=4, inner sep=0pt, draw, minimum size=6pt, xshift=4cm, thick] (v1) at (225:1.7557) {};
\draw[xshift=4cm] (1.24,1.24) -- (225:1.7557);

\draw[xshift=4cm] (1.24,1.24) -- (1.17557,4.0-1.628);
\draw[xshift=4cm] (1.24,1.24) -- (4.0-1.618,1.17557);
\draw[xshift=4cm] (1.24,1.24) -- (-4.0+1.618,1.17557);

\draw[xshift=4cm] (225:1.7557) -- (-4.0+1.618,-1.17557);
\draw[xshift=4cm] (225:1.7557) -- (4.0-1.618,-1.17557);
\draw[xshift=4cm] (225:1.7557) -- (-1.17557,4.0-1.628);

\node[shape=circle, xshift=4cm] (1) at (45:2.2) {\LARGE$v_1$} ;
\node[shape=circle, xshift=4cm] (2) at (225:2.2) {\LARGE$v_2$} ;

\end{tikzpicture}
\end{subfigure}
\hspace{0.5cm}
\begin{subfigure}[b]{0.3\linewidth}
\begin{tikzpicture}[scale=0.4, transform shape]

\draw[xshift=0.0] (1.618,1.17557) \foreach \x in {108,180,252,324} {
-- (\x:2)};
\foreach \x in {252,324,36,108,180}
\draw[fill] (\x:2) circle (2pt);

\draw[xshift=0.0] (-2,0) \foreach \x in {324,108,252,36,180} {
-- (\x:2)};

\draw[xshift=8cm] (-1.618,-1.17557) \foreach \x in {288,360,72,144} {
-- (\x:2)};
[xshitf=8cm]\foreach \x in {72,144,216,288,360}
\draw[fill, xshift=8cm] (\x:2) circle (2pt);

\draw[xshift=8cm] (2,0) \foreach \x in {144,288,72,216,360} {
-- (\x:2)};

\draw[xshift=4cm, yshift=4cm] (1.17557,-1.618) \foreach \x in {-54, 18, 90, 162, 234} {
-- (\x:2)};
[xshitf=4cm, yshift=4cm]\foreach \x in {-54, 18, 90, 162, 234}
\draw[fill, xshift=4cm, yshift=4cm] (\x:2) circle (2pt);
\draw[xshift=4cm, yshift=4cm] (1.17557,-1.618) \foreach \x in {234,90,-54} {
-- (\x:2)};
\draw[xshift=4cm, yshift=4cm] (1.17557,-1.618) (18:2)-- (162:2);

\node[shape=star , star points=4, inner sep=0pt, draw, minimum size=6pt, xshift=4cm, thick] (v1) at (45:1.7557) {};
\node[shape=star , star points=4, inner sep=0pt, draw, minimum size=6pt, xshift=4cm, thick] (v1) at (225:1.7557) {};

\draw[xshift=4cm] (1.17557,4.0-1.628) -- (1.24,1.24) -- (1.9021, 0.618+4);
\draw[xshift=4cm] (-1.17557,4.0-1.628) -- (225:1.7557) -- (-1.9021, 0.618+4);

\draw[xshift=4cm] (1.24,1.24) -- (4.0-1.618,1.17557);
\draw[xshift=4cm] (1.24,1.24) -- (-4.0+1.618,1.17557);

\draw[xshift=4cm] (225:1.7557) -- (-4.0+1.618,-1.17557);
\draw[xshift=4cm] (225:1.7557) -- (4.0-1.618,-1.17557);

\node[shape=circle, xshift=4cm] (1) at (43:2.2) {\LARGE$v_1$} ;
\node[shape=circle, xshift=4cm] (2) at (225:2.2) {\LARGE$v_2$} ;

\end{tikzpicture}
\end{subfigure}
\vspace{0.5cm}
\begin{subfigure}[b]{0.4\linewidth}
\begin{tikzpicture}[scale=0.4, transform shape]

\draw[xshift=0.0] (1.618,1.17557) \foreach \x in {108,180,252,324} {
-- (\x:2)};
\foreach \x in {252,324,36,108,180}
\draw[fill] (\x:2) circle (2pt);

\draw[xshift=0.0] (-2,0) \foreach \x in {324,108,252,36,180} {
-- (\x:2)};

\draw[xshift=8cm] (-1.618,-1.17557) \foreach \x in {288,360,72,144} {
-- (\x:2)};
[xshitf=8cm]\foreach \x in {72,144,216,288,360}
\draw[fill, xshift=8cm] (\x:2) circle (2pt);

\draw[xshift=8cm] (2,0) \foreach \x in {144,288,72,216,360} {
-- (\x:2)};

\draw[xshift=4cm, yshift=4cm] (1.17557,-1.618) \foreach \x in {-54, 90, 234, 18, 162, -54} {
-- (\x:2)};
[xshitf=4cm, yshift=4cm]\foreach \x in {-54, 18, 90, 162, 234}
\draw[fill, xshift=4cm, yshift=4cm] (\x:2) circle (2pt);

\draw[xshift=4cm, yshift=4cm] (1.17557,-1.618) (18:2)-- (90:2) -- (162:2);
\draw[xshift=4cm, yshift=4cm] (1.17557,-1.618) (-54:2)-- (234:2);

\node[shape=star , star points=4, inner sep=0pt, draw, minimum size=6pt, xshift=4cm, thick] (v1) at (45:1.7557) {};
\node[shape=star , star points=4, inner sep=0pt, draw, minimum size=6pt, xshift=4cm, thick] (v1) at (225:1.7557) {};

\draw[xshift=4cm] (1.17557,4.0-1.628) -- (1.24,1.24) -- (1.9021, 0.618+4);
\draw[xshift=4cm] (-1.17557,4.0-1.628) -- (225:1.7557) -- (-1.9021, 0.618+4);

\draw[xshift=4cm] (1.24,1.24) -- (4.0-1.618,1.17557);
\draw[xshift=4cm] (1.24,1.24) -- (-4.0+1.618,1.17557);

\draw[xshift=4cm] (225:1.7557) -- (-4.0+1.618,-1.17557);
\draw[xshift=4cm] (225:1.7557) -- (4.0-1.618,-1.17557);
\node[shape=circle, xshift=4cm] (1) at (42:2.2) {\LARGE$v_1$} ;
\node[shape=circle, xshift=4cm] (2) at (225:2.2) {\LARGE$v_2$} ;

\end{tikzpicture}
\end{subfigure}
\hspace{0.5cm}
\begin{subfigure}[b]{0.3\linewidth}
\begin{tikzpicture}[scale=0.4, transform shape]

\draw[xshift=0.0] (1.618,1.17557) \foreach \x in {108,180,252,324} {
-- (\x:2)};
\foreach \x in {252,324,36,108,180}
\draw[fill] (\x:2) circle (2pt);

\draw[xshift=0.0] (-2,0) \foreach \x in {324,108,252,36,180} {
-- (\x:2)};

\draw[xshift=8cm] (-1.618,-1.17557) \foreach \x in {288,360,72,144} {
-- (\x:2)};
[xshitf=8cm]\foreach \x in {72,144,216,288,360}
\draw[fill, xshift=8cm] (\x:2) circle (2pt);

\draw[xshift=8cm] (2,0) \foreach \x in {144,288,72,216,360} {
-- (\x:2)};

\draw[xshift=4cm, yshift=4cm] (-1.17557,1.618) \foreach \x in {-18, -162, 54, -90, -234, -18} {
-- (\x:2)};
\foreach \x in {-234, -162, -90, -18, 54}
\draw[fill, xshift=4cm, yshift=4cm] (\x:2) circle (2pt);

\draw[xshift=4cm, yshift=4cm] (-1.17557,1.618) \foreach \x in {-18, 54, -234, -162} {
-- (\x:2)};

\node[shape=star , star points=4, inner sep=0pt, draw, minimum size=6pt, xshift=4cm, thick] (v1) at (45:1.7557) {};
\node[shape=star , star points=4, inner sep=0pt, draw, minimum size=6pt, xshift=4cm, thick] (v1) at (225:1.7557) {};

\draw[xshift=4cm, yshift=4cm] (-90:2) -- (1.24,1.24-4) -- (-18:2);

\draw[xshift=4cm, yshift=4cm] (-90:2) -- (-1.24,-1.24-4) -- (-162:2);

\draw[xshift=4cm] (1.24,1.24) -- (4.0-1.618,1.17557);
\draw[xshift=4cm] (1.24,1.24) -- (-4.0+1.618,1.17557);

\draw[xshift=4cm] (225:1.7557) -- (-4.0+1.618,-1.17557);
\draw[xshift=4cm] (225:1.7557) -- (4.0-1.618,-1.17557);

\node[shape=circle, xshift=4cm] (1) at (41:2.2) {\LARGE$v_1$} ;
\node[shape=circle, xshift=4cm] (2) at (225:2.2) {\LARGE$v_2$} ;
\end{tikzpicture}
\end{subfigure}
 	\caption{Four 4-regular 2-connected graphs with $n=17$ that are not $1-$tough.}
 	\label{fig:T17_2_R5}
 \end{figure}
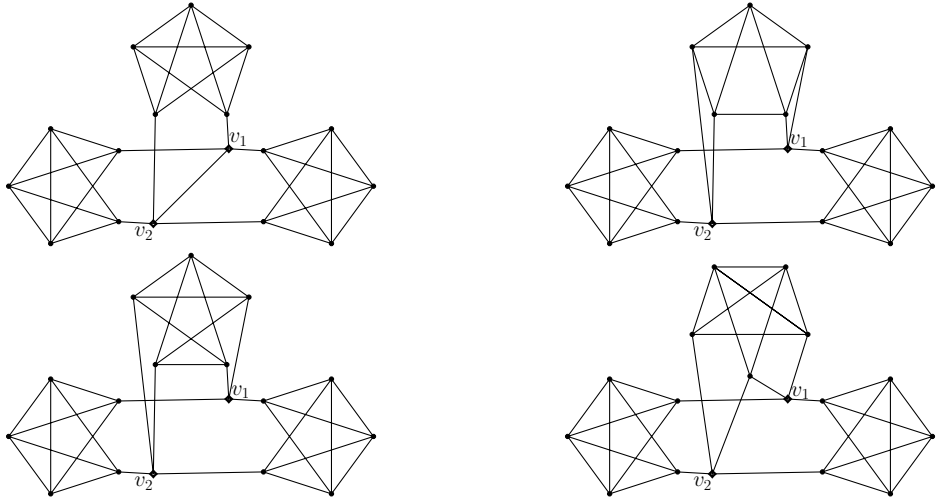
\begin{figure}[htbp]
  \vspace{-0.5cm}
\centering
\begin{subfigure}[b]{0.35\linewidth}
\begin{tikzpicture}[scale=0.4, transform shape]

\draw[xshift=0.0] (1.618,1.17557) \foreach \x in {108,180,252,324} {
-- (\x:2)};
\foreach \x in {252,324,36,108,180}
\draw[fill] (\x:2) circle (2pt);

\draw[xshift=0.0] (-2,0) \foreach \x in {324,108,252,36,180} {
-- (\x:2)};

\draw[xshift=8cm] (-1.732,-1.0) \foreach \x in {270,330,30,90,150} {
-- (\x:2)};
[xshitf=8cm]\foreach \x in {270,330,30,90,150,210}
\draw[fill, xshift=8cm] (\x:2) circle (2pt);

\draw[xshift=8cm] (-1.732,-1.0) \foreach \x in {90,330,150,270,30,210} {
-- (\x:2)};

\draw[xshift=4cm] (1.24,1.24) \foreach \x in {45, 135,225,315,45} {
-- (\x:1.7557)};
[xshitf=4cm]\foreach \x in {45, 135,225,315}
\draw[fill, xshift=4cm] (\x:1.7557) circle (2pt);

\draw[xshift=4cm] (1.24,1.24) -- (225:1.7557);
\draw[xshift=4cm] (1.24,-1.24) -- (135:1.7557);

\node[shape=star , star points=4, inner sep=0pt, draw, minimum size=6pt, xshift=4cm, thick] (v1) at (0,3) {};
\node[shape=circle, xshift=4cm] (1) at (0,3.4) {\LARGE$v_1$} ;
\draw[xshift=4cm] (0,3) -- (45:1.7557);
\draw[xshift=4cm] (0,3) -- (135:1.7557);

\node[shape=star , star points=4, inner sep=0pt, draw, minimum size=6pt, xshift=4cm, thick] (v1) at (0,-3) {};
\node[shape=circle, xshift=4cm] (2) at (0,-3.4) {\LARGE$v_2$} ;
\draw[xshift=4cm] (0,-3) -- (-45:1.7557);
\draw[xshift=4cm] (0,-3) -- (-135:1.7557);

\draw[xshift=4cm] (0,-3) -- (-4.0+1.618,-3.0+1.8);
\draw[xshift=4cm] (0,-3) -- (4.0-1.732,-1.0);

\draw[xshift=4cm] (0,3) -- (-4.0+1.618,3.0-1.8);
\draw[xshift=4cm] (0,3) -- (4.0-1.732,1.0);

\end{tikzpicture}
\end{subfigure}
\hspace{0.5cm}
\begin{subfigure}[b]{0.32\linewidth}
\begin{tikzpicture}[scale=0.4, transform shape]

\draw[xshift=0.0] (1.618,1.17557) \foreach \x in {108,180,252,324} {
-- (\x:2)};
\foreach \x in {252,324,36,108,180}
\draw[fill] (\x:2) circle (2pt);

\draw[xshift=0.0] (-2,0) \foreach \x in {324,108,252,36,180} {
-- (\x:2)};

\draw[xshift=8cm] (-1.732,-1.0) \foreach \x in {270,330,30,90,150} {
-- (\x:2)};
[xshitf=8cm]\foreach \x in {270,330,30,90,150,210}
\draw[fill, xshift=8cm] (\x:2) circle (2pt);

\draw[xshift=8cm] (-1.732,1.0) \foreach \x in {30,270,90,330,210,150} {
-- (\x:2)};

\draw[xshift=4cm] (1.24,1.24) \foreach \x in {45, 135,225,315,45} {
-- (\x:1.7557)};
[xshitf=4cm]\foreach \x in {45, 135,225,315}
\draw[fill, xshift=4cm] (\x:1.7557) circle (2pt);

\draw[xshift=4cm] (1.24,1.24) -- (225:1.7557);
\draw[xshift=4cm] (1.24,-1.24) -- (135:1.7557);

\node[shape=star , star points=4, inner sep=0pt, draw, minimum size=6pt, xshift=4cm, thick] (v1) at (0,3) {};
\node[shape=circle, xshift=4cm] (1) at (0,3.4) {\LARGE$v_1$} ;
\draw[xshift=4cm] (0,3) -- (45:1.7557);
\draw[xshift=4cm] (0,3) -- (135:1.7557);

\node[shape=star , star points=4, inner sep=0pt, draw, minimum size=6pt, xshift=4cm, thick] (v1) at (0,-3) {};
\node[shape=circle, xshift=4cm] (2) at (0,-3.4) {\LARGE$v_2$} ;
\draw[xshift=4cm] (0,-3) -- (-45:1.7557);
\draw[xshift=4cm] (0,-3) -- (-135:1.7557);

\draw[xshift=4cm] (0,-3) -- (-4.0+1.618,-3.0+1.8);
\draw[xshift=4cm] (0,-3) -- (4.0-1.732,-1.0);

\draw[xshift=4cm] (0,3) -- (-4.0+1.618,3.0-1.8);
\draw[xshift=4cm] (0,3) -- (4.0-1.732,1.0);

\end{tikzpicture}
\end{subfigure}
 	\caption{Two 4-regular 2-connected graphs with $n=17$ that are not $1-$tough.}
	\label{fig:T17_2_G6}
 \end{figure}
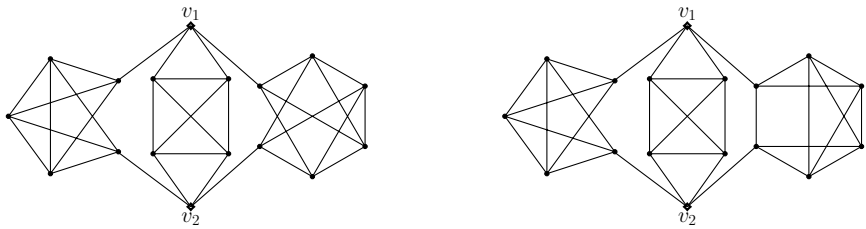
\begin{figure}[htbp]
\centering
\begin{minipage}[t]{0.45\linewidth}
\centering
\begin{tikzpicture}[scale=0.4, transform shape]

\draw[xshift=0.0] (1.618,1.17557) \foreach \x in {108,180,252,324} {
-- (\x:2)};
\foreach \x in {252,324,36,108,180}
\draw[fill] (\x:2) circle (2pt);

\draw[xshift=0.0] (-2,0) \foreach \x in {324,108,252,36,180} {
-- (\x:2)};

\draw[xshift=8cm] (-1.618,-1.17557) \foreach \x in {288,360,72,144} {
-- (\x:2)};
[xshitf=8cm]\foreach \x in {72,144,216,288,360}
\draw[fill, xshift=8cm] (\x:2) circle (2pt);

\draw[xshift=8cm] (2,0) \foreach \x in {144,288,72,216,360} {
-- (\x:2)};

[xshitf=4cm]\foreach \x in {45, 135,225,315}
\node[shape=star , star points=4, inner sep=0pt, draw, minimum size=6pt, xshift=4cm, thick] (v1) at (\x:1.7557) {};

\node[shape=circle, xshift=4cm] (1) at (45:2.2) {\LARGE$v_1$};
\node[shape=circle, xshift=4cm] (1) at (135:2.2) {\LARGE$v_2$};
\node[shape=circle, xshift=4cm] (1) at (225:2.2) {\LARGE$v_3$};
\node[shape=circle, xshift=4cm] (1) at (315:2.2) {\LARGE$v_4$};

\draw[fill, xshift=4cm] (0,2) circle (2pt);
\draw[fill, xshift=4cm] (0,0) circle (2pt);
\draw[fill, xshift=4cm] (0,-2) circle (2pt);

[xshitf=4cm]\foreach \x in {45, 135,225,315}
\draw[xshift=4cm] (0,2) -- (\x:1.7557);

[xshitf=4cm]\foreach \x in {45, 135,225,315}
\draw[xshift=4cm] (0,-2) -- (\x:1.7557);

[xshitf=4cm]\foreach \x in {45, 135,225,315}
\draw[xshift=4cm] (0,0) -- (\x:1.7557);

\draw[xshift=4cm] (225:1.7557) -- (-4.0+1.618,-3.0+1.8);
\draw[xshift=4cm] (315:1.7557) -- (4.0-1.618,-3.0+1.8);

\draw[xshift=4cm] (135:1.7557) -- (-4.0+1.618,3.0-1.8);
\draw[xshift=4cm] (45:1.7557) -- (4.0-1.618,3.0-1.8);

\end{tikzpicture}
\caption{A 4-regular 2-connected graph with $n=17$  that is not $1-$tough.}\label{fig:T17_4}
\end{minipage}
\hspace{-1cm}
\begin{minipage}[t]{0.45\linewidth}
\centering
\begin{tikzpicture}[scale=0.4, transform shape]
\tikzset{cross/.style={cross out, draw=black, minimum size=2*(#1-\pgflinewidth), inner sep=0pt, outer sep=0pt},
cross/.default={1pt}}

\draw[xshift=0.0] (1.618,1.17557) \foreach \x in {108,180,252,324} {
-- (\x:2)};
\foreach \x in {252,324,36,108,180}
\draw[fill] (\x:2) circle (2pt);

\draw[xshift=0.0] (-2,0) \foreach \x in {324,108,252,36,180} {
-- (\x:2)};
\end{tikzpicture}
\caption{  The graph $R_5$.}
\label{fig:R5}
\end{minipage}
\end{figure}

A similar analysis for $n=17$ allows to identify seven $\varphi$-graphs with $n=17$. Indeed, when solving problem $\mathcal{Q}_{n,k}$ for $n = 17$, it turns out that $v'(17,k) > k$  for $k=2, 4$  with optimal values, respectively, $v'(17,2) = 3$ and $v'(17,4) = 5$. In particular, one optimal solution of problem $\mathcal{Q}_{17,2}$ is $x^*_5 = 3$ and $x^*_i=0$ for $i\not= 5$. This solution determines the four graphs in Figure \ref{fig:T17_2_R5}. By solving again $\mathcal{Q}_{17,2}$ with the additional constraint $x_3+x_4\geq 1$ one obtains a different solution $\bar x$ of value 3 with $\bar{x}_4=\bar{x}_5=\bar{x}_6=1$ and $\bar{x}_i=0$ for $i\not=4,5,6$. This solution is compatible only with the two graphs  in Figure \ref{fig:T17_2_G6}. The solutions $x^*$ and $\bar{x}$ are the only solutions of $\mathcal{Q}_{17,2}$ of value 3. Indeed, by adding to $\mathcal{Q}_{17,2}$ either the constraint $x_3\geq 1$ or the constraint  $\sum_{i\geq 7} x_i \geq 1$ one obtains 2 as optimal value. Finally, the optimal solution found when solving $\mathcal{Q}_{17,4}$ is $\hat x$ with   $\hat x_1=3$, $\hat x_5=2$ and $\hat x_i=0$, $i\not=1,5$. It has value $5$ and corresponds to the graph in Figure \ref{fig:T17_4}. By solving again $\mathcal{Q}_{17,4}$  with the  additional constraint either $x_1\leq 2$ or $x_1 \geq 4$ or   $x_5\leq1$, one always obtains an optimal value at  most $4$. So $\hat x$ is the unique optimal solution of problem $\mathcal{Q}_{17,4}$.

We observe that all the eight graphs determined through the preliminary analysis contain one or more subgraphs isomorphic to the graph  in Figure \ref{fig:R5} which is obtained by removing one edge from the complete graph $K_5$. We call this graph $R_{5}$. This implies that any 2-connected 4-regular graph with $n = 16$ or $n = 17$ that does not contain any $R_5$ is 1-tough. This fact will play a role in reducing the computations needed to prove that $f(1t, \lnot H) = 18$. 

The preliminary analysis has been performed also for  $18 \leq n \leq 20$.  For all these cases, we found that by removing $k$ vertices the graph gets disconnected in at most $k+1$ components. Defining $K(G) := \{k : v'(n,k) > k\}$ and denoting by $r(G)$
the number of disjoint subgraphs of $G$ isomorphic to $R_{5}$, the results obtained for $16 \leq n \leq 20$ can be summarized as follows
\begin{itemize}
\item if $n = 16$ then $r(G) = 2 \wedge K(G) = \{2\}$;
\item if $n = 17$ then $(r(G) = 1 \wedge K(G) = \{2\}) \vee (r(G) = 2 \wedge K(G) \subseteq \{2,4\})$;
\item if $n = 18$ then $(r(G) = 1 \wedge K(G) \subseteq \{2, 4\}) \vee (r(G) \in \{0, 2\} \wedge K(G) = \{2\})$;

\item if $n = 19$ then  
$(r(G) \in \{0, 1\} \wedge K(G) \subseteq \{2, 4\}) \vee (r(G) = 2  \wedge K(G) \subseteq \{2,5\})$;
\item if $n = 20$ then  $$(r(G) \in \{0, 2\} \wedge K(G) \subseteq \{2, 4\}) \vee (r(G) = 1  \wedge K(G) \subseteq \{2, 4, 5\}) \vee (r(G) = 3  \wedge K(G) = \{3\}).$$
\end{itemize}

%

\section{The ILP model for solving problem $P(n,\cC,\lnot \cH)$} 
\label{sec:model}
Let us consider the problems $P(n,\cC, \lnot \cH)$ defined in the introduction where 
\begin{align*}
\mathcal{C} \in& ~\{\text{connectivity},\text{2-connectivity}, 1\text{-toughness}\}, \\
\mathcal{H} \in& ~\{\text{hamiltonicity, homogeneous traceability, traceability}\}.
\end{align*}
We formulate any problem $P(n,\cC,\lnot \cH)$ as an ILP feasibility model $\genericmodel$ whose feasible solutions correspond to those 4-regular graphs with $n$ vertices that satisfy property $\cC$ and do not satisfy property $\cH$. As a consequence, the value  $f(\cC, \lnot \cH)$ corresponds to the minimum  $n$ for which model $\genericmodel$ has a feasible solution.
More in detail, let $K^{n} = (V^{n}, E^{n})$ denote the complete graph with  $V^{n} = \{1, \dots, n\}$. We introduce a binary variable $x_{e}$ for each edge $e \in E^{n}$ and 
associate to any $x \in \{0,1\}^{|E^{n}|}$ the graph $G(x) = (V^{n}, E(x))$ with edge-set $E(x) = \{e \in E^{n}: x_{e} = 1\}$. The linear inequalities of the model are the following. The condition that $G(x)$ is 4-regular is imposed by the family of $n$ constraints, called {\em degree constraints},
\begin{equation}
\sum_{e\in\partial(\{i\})} x_e = 4 \qquad \qquad  \forall~i \in V^{n}\label{eq:4R}
\end{equation}
Let $\cH^{n}$ and $\cP_{\ell}(i)$, $i \in V^{n}$, denote the set of the Hamilton cycles and, respectively, the set of the  paths of length $\ell$ starting at $i$ in $K^{n}$. Hence $\cP_{n-1}(i)$ denotes the set of the Hamilton paths starting at vertex $i$. The condition that the graph $G(x)$ is not hamiltonian can be imposed by  the family of constraints
\begin{equation} 
\sum_{e\in H } x_e \leq |H|-1 = n-1 \qquad  \forall~H \in\mathcal{H}^{n} \label{eq:NH}.
\end{equation}
Similarly, the condition that $G(x)$ is not  is not traceable can be imposed by the family of constraints
\begin{equation} 
\sum_{e\in P } x_e \leq |P|-1 = n- 2 \qquad \forall~i \in V^{n}, ~P \in \cP_{n-1}(i) \label{eq:NoPli}.
\end{equation}
The condition that $G(x)$ is not homogeneously traceable may be imposed by requiring that for a particular vertex, let us say vertex 1, 
\begin{equation} 
\sum_{e\in P } x_e \leq |P|-1 = n-2  \qquad \forall~P \in \cP_{n-1}(1) \label{eq:NoPl}.
\end{equation}
Finally, let us consider how to model each connectivity condition  $\cC$.
The family of contraints
\begin{equation} 
\sum_{e\in \partial(S) } x_e \geq 1 \qquad\qquad \forall~S \subset V^{n},\ S \neq \emptyset \label{eq:conn}
\end{equation}
guarantees that the graph $G(x)$ is connected while the family of contraints 
\begin{equation}
\sum_{e\in \partial(S)\setminus \partial(\{i\})} x_e \geq 1 \qquad \forall\ i\in V^{n}, ~S\subset V^{n}\setminus\{i\},\ S \neq \emptyset \label{eq:2C}
\end{equation}
guarantees that any graph obtained from $G(x)$ by removing any vertex is connected, i.e.,  that $G(x)$ is 2-connected. Finally, the 1-toughness of $G(x)$ may be imposed by introducing the family of constraints
\begin{equation} 
\sum_{1\le a < b \le t}\,\,\sum_{e\in\partial(W_a)\,\cap\,\partial(W_b)} x_e \geq 1 \quad \forall \text{ partition } S, W_1, \dots, W_t \text{ of } V^{n} \text{ with } t > |S|. \label{eq:1T}
\end{equation}
Each model $\genericmodel$ is defined by the degree contraints and the families of constraints that impose condition $\cC$ and forbid property $\cH$. 
We observe that all the previous families,  except that of the degree constraints, contain a number of inequalities which is exponential with respect to $n$. Since this number is very high even for $n=16$, we have adopted a cutting plane approach to generate and add these constraints to the model only when needed. This requires to be able to  solve the  {\em separation problem} corresponding to each family of constraints.

\subsection{The separation problems} 
\label{subsec:separation}
The separation problem with respect to a family $\cal L$ of inequalities is the following: given  a solution $\bar x$ (not necessarily integer), find  an inequality of $\cal L$ violated by $\bar x$ or determine that such inequality does not exist. An algorithm for this problem is called a separation algorithm for $\cal L$.

\subsubsection*{The separation problem for the $\lnot \cH$-constraints}
\label{subsubsec:-H}
A solution $\bar x$ violates the non-hamiltonicity constraints (\ref{eq:NH})  if and only if there exists a Hamilton cycle $H \in \mathcal{H}^{n}$ such that 
\begin{equation*}
\sum_{e \in H } \bar x_{e} > n-1 \Leftrightarrow  \sum_{e \in H } (\bar x_{e} - 1) > -1 \Leftrightarrow  \sum_{e \in H } (1 - \bar x_{e}) < 1.
\end{equation*}
As a consequence, the separation problem for constraints (\ref{eq:NH}) can be solved by finding the shortest Hamilton cycle in $K^{n}$ with respect to the lengths $c_{e} := 1 - \bar x_{e}$ for each $e \in E^{n}$. This is  a  Traveling Salesman Problem (TSP). If the optimal TSP solution $H^*$ has value smaller than 1 then the constraint
\begin{equation*}
\sum_{e \in H^*} x_{e} \leq n - 1
\end{equation*}
is violated by $\bar x$ and the constraint is added to the model. Otherwise, $\bar x$ satisfies all constraints (\ref{eq:NH}).
Similarly, the constraints \ref{eq:NoPli} (respectively, \ref{eq:NoPl}) are satisfied if and only if the shortest path with respect to the lengths $c_{e}$ in ${\cal P}_{n-1}$ (respectively, in  ${\cal P}_{n-1}(i)$)  has length strictly less than 1.    
It is well known that the TSP and the problem of finding the shortest path with a given number of edges are NP-hard. However, in our application $n$ is fixed and rather small. Solving these problems  on such small graphs is quite simple and there are several effective algorithms to this end. In particular, we have used a simple branch-and-bound procedure.

\subsubsection*{The separation problem for the connectivity and 2-connectivity constraints}
A solution $\bar x$ does not satisfy the 2-connectivity constraints (\ref{eq:2C}) if and only for some vertex $i \in V^{n}$ and some subset $\bar S \subset V^{n} \setminus \{i\}$, $\bar S \neq \emptyset$,  the sum $\sum_{e \in \partial(\bar S)}\bar x_{e}$ over the cut $\partial(\bar S)$ in $K^{n}[V^{n}\setminus\{i\}]$ is strictly smaller than 1. Thus the separation problem for the 2-connectivity constraints may be solved by finding,  for each $i \in V^{n}$,
a minimum-cut on the graph $K^{n}[V^{n}\setminus\{i\}]$ with respect to the weights $w_{e} :=\bar x_{e}$. 
If for some $i$ the optimal value is smaller than 1, the 2-connectivity inequality defined by $i$ and the optimal solution $\bar S$ is violated by $\bar x$, otherwise $\bar x$ satisfies all inequalities (\ref{eq:2C}).  If $\bar x$ is integer, the separation problem can be alternatively solved in time $O(n)$ by searching for the articulation points of the graph $G(\bar x)$, i.e., the vertices whose removal disconnects the graph \cite{S98}. Similar approaches  may be used to separate the connectivity constraints (\ref{eq:conn}) for fractional and integer solutions.  

\subsubsection*{The separation problem for the 1-toughness constraints}
The separation problem with respect to the 1-toughness constraints (\ref{eq:1T}) is NP-hard. We separate these constraints only for binary solutions $\bar x$ that determine a  4-regular and 2-connected graph $G(\bar x)$.  
For given $k$ and $t$, with $k < t$, we look for a partition $S, W_1, \dots, W_t$ of $V^{n}$ such that $|S| = k$ and $W_{1}, \dots, W_{t}$ are the vertex sets of the connected components of the graph 
$G(\bar x)[V^{n}\setminus S]$  by solving the following ILP problem. Let  $z_{i}$, $i \in V^{n}$, and $y_{ir}$, $i \in V^{n}$ and $r = 1, \dots, t$, be binary variables such that
\begin{equation*}
    z_i = \begin{cases} 1 & \text{if~} i\in S \\ 0 & \text{otherwise} \end{cases} \qquad \qquad \qquad
    y_{ir} = \begin{cases} 1 &  \text{if~} i\in W_r \\ 0 & \text{otherwise} \end{cases}
\end{equation*}
Then the following constraints are satisfied only by partitions with the required properties
\begin{eqnarray}
\label{eq2}         \mathcal{T}_{k,t}: ~~~~~~	~~~~~~	\sum_{r=1}^{t} y_{ir} + z_{i}&=& 1 \quad\quad\quad\quad \forall~i\in V^{n} \\
\label{eq3} 		y_{ir} &\leq& y_{jr}+z_j \quad\quad\forall~\{i,j\}\in E(\bar x),~~ r=1,\dots,t\\
\label{eq4} 		y_{jr} &\leq& y_{ir}+z_i \quad\quad\forall \{i,j\}\in E(\bar x),~~ r=1,\dots,t \\
\label{eq5} 		\sum_{i \in V^{n}} y_{ir} &\geq&  1 \quad\quad\quad \forall~r=1,\dots, t\\
\label{eq1} 		\sum_{i\in V^{n}} z_i &=& k \\
\label{eq9} 		z_i &\in& \{0,1\} \quad \forall~i\in V^{n} \\
\label{eq10} 		y_{ir} &\in& \{0,1\} \quad \forall~i\in V^{n}, r=1.,\dots,t.\\
\end{eqnarray}
Conditions (\ref{eq2}) impose that each vertex belongs to exactly one set of the partition. Conditions (\ref{eq3}) and (\ref{eq4}) guarantee that each edge $e \in E(\bar x)$ has either both endpoints in a same set $W_{r}$ or at least one endpoint in $S$.  Finally, all the sets  $W_{r}$ are not empty  by  constraints (\ref{eq5}) and $|S| = k$ by constraint (\ref{eq1}).

\subsection{The branch-and-cut procedure}
\label{sec:bac}

As already remarked, each model ${\cal M}(n, \cC, \lnot \cH)$  has exponential size with respect to $n$ and must therefore be solved with a constraint-generation approach.
The standard way to do  this is called {\em branch-and-cut}. Branch-and-Cut is a version of branch and bound in which the constraint matrix at each node $N$ of the search tree contains only a (small) subset of the constraints of the original model, while some of the missing constraints may be added at run time.  Let us denote by  $\mathcal{M}(N)$ the set of constraints of the subproblem corresponding to  node $N$. These are the constraints that were input at the root node, plus the branching constraints (fixing variables to 0 or 1) and  all  the constraints which were added in the nodes on the path from the root to $N$.

Whenever the LP-relaxation of $\mathcal{M}(N)$ is solved, yielding a solution $\bar x$, the feasibility of $\bar x$ with respect to $\genericmodel$  must be checked. 
The solution $\bar x$ could be infeasible either because it is fractional, or because it violates some of the constraints of $\genericmodel$ which are missing at $N$. In order to find which constraints, if any, are not satisfied by $\bar x$,  we  first run the separation algorithm described in Section \ref{subsec:separation} to possibly find
one of the constraints enforcing $\lnot \cal H$ which is violated. If this is not the case and $\bar x$ is integer, we also run the separation algorithms  which 
check if one of the constraints enforcing property $\cal C$ is violated. The properties of connectivity, 2-connectivity and 1-toughness are checked in this order.
If we find any violated constraints,  we add them  to $\mathcal{M}(N)$ and solve the problem again. This phase is called constraint- (or cut-) generation.

The processing of the node terminates only when $\bar x$ is integer and feasible  for $\genericmodel$, or when $\bar x$ is fractional
but satisfies all constraints imposing
 property $\lnot {\cal H}$. 
If $\bar x$ is feasible,  it induces a graph $G(\bar x)$ with the sought properties and the search is terminated with a positive answer to problem
${\cal P}(n, \cC, \lnot \cH)$.
Otherwise,  a branching is performed from $N$, by picking a fractional component $\bar x_j$ and creating two new subproblems, $N'$ in which we fix  $x_j=0$, and $N''$ in which we fix $x_j=1$.

\subsection{Implementation decisions}\label{sec:decision}
To conclude the description of the ILP model, we briefly describe three implementation decisions that we have taken in order to speed-up the search. 

\subsubsection{Symmetries and orbital branching}

Let us consider the generic ILP model $\genericmodel$, a solution $x \in \{0,1\}^{|E^{n}|}$ and the associated graph $G(x)$. For every permutation $\pi\in S_n$ we can define a new solution $\pi(x)$ by setting $\pi(x)_{ij} = x_{\pi(i)\pi(j)}$ for each $\{i,j\} \in E^{n}$. The graph $G(\pi(x))$ is clearly isomorphic to $G(x)$. Since 
 all the connectivity properties and the  hamiltonian properties that we are considering are preserved by graph isomorphisms, $x$ is feasible for $\genericmodel$ if and only if  $\pi(x)$ is. This implies that every permutation $\pi\in S_n$ induces a symmetry of the model, i.e., the model has many different, but in fact isomorphic, solutions. It is well known that even relatively small instances of ILP problems with large groups of symmetries can be extremely difficult to solve via branch and cut. For this reason several techniques have been proposed in the literature to reduce the impact of symmetries (see for instance the surveys of Margot \cite{M10} and Pfetsch and Rehn \cite{PR19}). Among these techniques, a very effective one is Orbital Branching by Ostrovski and al. \cite{OLRS11}.
 
 The orbital branching method  requires to compute at each node $N$ of the branch and bound tree the group ${\cal G}^{N}$ of the permutations of $S_{n}$ that stabilizes the sets $B_{0}(N)$ and $B_{1}(N)$ of the indices of the variables that have been fixed at 0 and 1 at $N$ (because of branching or some other reason). The orbit of an edge $\bar e$ under the action of ${\cal G}^{N}$ is the set  $O(\bar e) = \{\pi(\bar e) : \pi \in {\cal G}^{N}\}$.  The main idea of orbital branching is that, given a free variable $x_{\bar e}$, we can create two new nodes in the branch and bound tree based on the disjunction $(x_{\bar e} = 1)\vee (\sum_{e\in O(\bar e)} x_{e} =0 )$. 
The orbital branching effectiveness can be strengthened by using a fixing technique introduced in \cite{OLRS11}. In orbital branching we ensure that any two nodes are not equivalent with respect to the symmetries found at their first common ancestor. It is possible, however, that two child subproblems are equivalent with respect to a symmetry group found elsewhere in the tree. In order to overcome this situation, orbital fixing  works as follows.
Let $I_0$ and  $I_1$ be the index sets of variables  fixed to 0 and, respectively, to 1 at the root node.  Note that $I_0\subseteq B_0(N)$ and $I_1\subseteq B_1(N)$ for each $N$. Given the group ${\cal G}(B_1(N),I_0)$ of the permutations in $S_{n}$ that stabilize the sets $B_1(N)$ and $I_0$, let $ O(e) $ denote the orbit of the edge $e$ under the action of  ${\cal G}(B_1(N),I_0)$.  Consider the set  $F_0 = \cup_{e \in B_{0}(N)}\; O(e)$ containing all the edges belonging to the orbits of  edges in $B_{0}(N)$. The results concerning orbital branching and orbital fixing guarantee that, given a free variable $x_{\bar e}\not\in F_{0}\cup B_{1}(N)$,  
two new nodes may be created according to the disjunction  
 \begin{equation*}
      \Big(x_{\bar e} =1 \;\wedge\; \sum_{e \in F_{0}} x_{e} = 0 \Big)\vee \Big(\sum_{e\in O(\bar e)} x_{e} =0 \;\wedge\; \sum_{e \in F_{0}} x_{e} = 0 \Big).
 \end{equation*}
 
 Clearly, the additional computational effort required  by the method to compute the groups of symmetries ${\cal G}^{N}$ and ${\cal G}(B_1(N),I_0)$ and their orbits is worthwhile as long as it returns orbits of rather large size, in which case the orbital branching and the orbital fixing rule significantly limit the visit of isomorphic solutions. Since the branching constraints tend to reduce the symmetries of the problem, orbital branching is usually performed only at the first levels of the branch and bound tree.

\subsubsection{Decomposition strategies for the solutions space} 
\label{subsec:decomposition}

For every model $\genericmodel$ we have adopted a decomposition of the set of feasible solutions based on the following idea. Given $k \leq  n$, we define the problem $Q(n, \cC, \lnot P_{k})$ as the problem of finding a graph satisfying property $\cC$ which (1) contains a  path of length $k-1$ and (2) does not contain a  path of length $k$ (a Hamilton  cycle if $k=n$). When $\cH = HT$ we require that these paths start at vertex 1.  Clearly, problem $P(n, \cC, \lnot \cH)$ is infeasible if and only if problem $Q(n, \cC, \lnot P_{k})$ is infeasible for every $k \leq r$, where $r = n$ if $\cH = H$ and $r = n-1$ if $\cH = T, HT$.  Since each connected 4-regular graph with $n \geq 15$ vertices contains at least one path of length 8, we  can start from the value $k = 8$. Condition (1) is imposed  by  fixing to 1 in the initial model the edges of a  path of length $k-1$ in $K^{n}$ (for instance the edges$\{h, h+1\}$ for $h=1, \dots, k-1$). Condition (2) can be guaranteed by imposing that the constraints (\ref{eq:NoPli}) are satisfied for each path in ${\cal P}_{k}(i)$ (in  ${\cal P}_{k}(1)$ if $\cH = HT$) instead than each path in ${\cal P}_{n-1}(i)$. The separation routine for the TSP-constraints has been easily adjusted to separate the modified constraints.  The use of this decomposition strategy allowed us to reduce by more than one order of magnitude the overall computational time required to solve model $\model{n}{1t}{H}$ with respect to the straightforward solution reported in \cite{LPR20}.

A second type of decomposition was adopted for the model $\model{n}{1t}{H}$ with $16 \leq n \leq 18$. A main concern in solving this model is that the separation routine for the 1-toughness constraints (\ref{eq:1T}) takes a considerable time.
In order to overcome this drawback,  by exploiting the preliminary analysis of Section \ref{sec:preliminary}, we have identified a small number of cases in which we actually do have to impose these constraints and, for these cases, which problems $\mathcal{T}_{k,t}$   have to be solved to separate the 1-toughness inequalities. In the remaining cases we can relax the constraints (\ref{eq:1T}) since they are implied by the 2-connectivity conditions. Let us now briefly describe this decomposition scheme. As it follows from the results of  Section \ref{sec:preliminary}, every 4-regular 2-connected graph with 16 and 17 vertices which is not 1-tough contains at least one subgraph isomorphic to $R_{5}$. For $n=18, 19, 20$ the graph may contain no $R_{5}$. Let $\mathcal{F}_{n}$ be the set of solutions of $\model{n}{1t}{H}$. We partition $\mathcal{F}_{n}$ into three sets, namely $\mathcal{F}_{n}(2R5)$, $\mathcal{F}_{n}(1R5)$ and  $\mathcal{F}_{n}(NOR5)$, which contain  the solutions corresponding to graphs with, respectively, at least two (disjoint) copies, a single copy or no copy of $R_{5}$. 

This partitioning allows us to fix many variables in the models in the first two cases. In particular, since the solutions in  $\mathcal{F}_{n}(2R5)$  contain two disjoint  copies of $R_{5}$, we can fix to 1 the variables corresponding to the edges shown in Fig. \ref{fig:fix12}(a). Similarly, for  the solutions in  $\mathcal{F}_{n}(1R5)$  we can fix to 1 the variables corresponding to the edges in Fig. \ref{fig:fix12}(b). Moreover, in the last two cases, we have to add to the model a set of inequalities which forbid the presence of any $R_5$ other than the one possibly fixed as above. These inequalities, called
{\em noR5-constraints (NOR5)}, are 
\begin{equation} \label{eq:NOR5}
\sum_{e\in E(V')} x_e \leq 8 \qquad \forall\ V'\subseteq W,~~|V'|=5,
\end{equation}
where  
$W=\{6,\ldots,n\}$ in the case of $\mathcal{F}_{n}(1R5)$,  
and $W=V^{n}$ in the case of $\mathcal{F}_{n}(NOR5)$. Since there are several hundred inequalities, we have decided not to add them all to the model, but to separate them only when needed. 
\begin{figure}[t]
\centering
\begin{subfigure}[b]{0.45\linewidth}
\centering
\begin{tikzpicture}[scale=0.40, transform shape]
\tikzset{cross/.style={cross out, draw=black, minimum size=2*(#1-\pgflinewidth), inner sep=0pt, outer sep=0pt},
cross/.default={1pt}}

\draw[xshift=0.0] (1.618,1.17557) \foreach \x in {108,180,252,324} {
-- (\x:2)};
\foreach \x in {252,324,36,108,180}
\draw[fill] (\x:2) circle (2pt);

\draw[xshift=0.0] (-2,0) \foreach \x in {324,108,252,36,180} {
-- (\x:2)};

\draw[xshift=4.5cm] (-1.618,-1.17557) \foreach \x in {288,360,72,144} {
-- (\x:2)};
[xshitf=8cm]\foreach \x in {72,144,216,288,360}
\draw[fill, xshift=4.5cm] (\x:2) circle (2pt);

\draw[xshift=4.5cm] (2,0) \foreach \x in {144,288,72,216,360} {
-- (\x:2)};

\node[shape=circle, above] (1) at (35:2) {\LARGE$4$} ;
\node[shape=circle, above] (1) at (108:2) {\LARGE$2$} ;
\node[shape=circle, left] (1) at (180:2) {\LARGE$1$} ;
\node[shape=circle, left] (1) at (252:2) {\LARGE$3$} ;
\node[shape=circle, above] (1) at (324:2) {\LARGE$5$} ;

\node[shape=circle, xshift=4.5cm, above] (1) at (144:2) {\LARGE$6$} ;
\node[shape=circle, xshift=4.5cm, above] (1) at (72:2) {\LARGE$8$} ;
\node[shape=circle, xshift=4.5cm, right] (1) at (360:2) {\LARGE$10$} ;
\node[shape=circle, xshift=4.5cm, right] (1) at (288:2) {\LARGE$9$} ;
\node[shape=circle, xshift=4.5cm, above] (1) at (216:2) {\LARGE$7$} ;
\end{tikzpicture}
\caption{}\label{fig:M1}  
\end{subfigure}
\hspace{-2cm}
\begin{subfigure}[b]{0.45\linewidth}
\centering
\begin{tikzpicture}[scale=0.4, transform shape]
\tikzset{cross/.style={cross out, draw=black, minimum size=2*(#1-\pgflinewidth), inner sep=0pt, outer sep=0pt},
cross/.default={1pt}}

\draw[xshift=0.0] (1.618,1.17557) \foreach \x in {108,180,252,324} {
-- (\x:2)};
\foreach \x in {252,324,36,108,180}
\draw[fill] (\x:2) circle (2pt);

\draw[xshift=0.0] (-2,0) \foreach \x in {324,108,252,36,180} {
-- (\x:2)};

\draw[fill, xshift=2cm] (36:2) circle (2pt);
\draw[fill, xshift=2cm] (324:2) circle (2pt);

\node[](1) at (36:2){};
\node[](2) at (108:2){};
\node[](3) at (180:2){};
\node[](4) at (252:2){};
\node[](5) at (324:2){};
\node[xshift=2cm](6) at (324:2){};
\node[xshift=2cm](7) at (36:2){};

\node[shape=circle, above] (a) at (36:2) {\LARGE$4$} ;
\node[shape=circle, above] (b) at (108:2) {\LARGE$2$} ;
\node[shape=circle, left] (c) at (180:2) {\LARGE$1$} ;
\node[shape=circle, left] (d) at (252:2) {\LARGE$3$} ;
\node[shape=circle, above] (e) at (324:2) {\LARGE$5$} ;

\node[shape=circle, xshift=2cm, above] (f) at (324:2) {\LARGE$7$} ;
\node[shape=circle, xshift=2cm, above] (g) at (36:2) {\LARGE$6$} ;

\path[-]  (1) edge (7);
\path[-]  (5) edge (6);

\end{tikzpicture}\caption{}\label{fig:M2}  
\end{subfigure}
\caption{Variables fixed to 1 for the solutions in $\mathcal{F}_{n}(2R5)$  (a)  and in $\mathcal{F}_{n}(1R5)$ (b).}
\label{fig:fix12}
\end{figure}
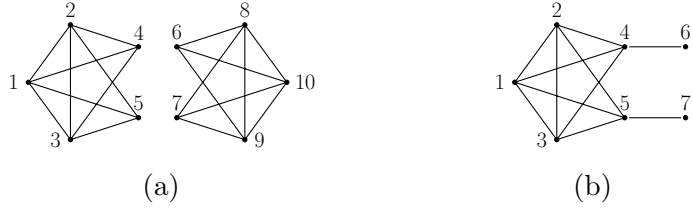
Based on the above decomposition, for $n=16, 17, 18$ we have solved model $\model{n}{1t}{H}$ three times: 
\begin{itemize}
    \item with feasible set $\mathcal{F}_{n}(NOR5)$ by removing  the 1-toughness constraints (\ref{eq:1T}) for $n \leq 17$ and  by solving problem $\mathcal{T}_{2,3}$ for $n=18$;
\item with feasible set $\mathcal{F}_{n}(1R5)$ by removing  the 1-toughness constraints when  $n=16$, by solving problem $\mathcal{T}_{2,3}$ when $n=17$ (to  exclude the two graphs in Fig. \ref{fig:T17_2_G6}) and by solving both the problems  $\mathcal{T}_{2,3}$ and  $\mathcal{T}_{4,5}$ when $n=18$;
\item with feasible set $\mathcal{F}_{n}(2R5)$ by solving problem $\mathcal{T}_{2,3}$ when $n=16$ (to exclude the graph in Fig. \ref{Figure:n16}) and when $n=18$ and both the problems  $\mathcal{T}_{2,3}$ and  $\mathcal{T}_{4,5}$ when $n=17$ (to exclude the graphs in Fig. \ref{fig:T17_2_R5} and in Fig. \ref{fig:T17_4}, respectively). 
\end{itemize}
When $n \geq 19$ the above decomposition is not as useful. Indeed, based on the preliminary analysis, one has to solve both problems $\mathcal{T}_{2,3}$ and  $\mathcal{T}_{4,5}$ also on the solution set $\mathcal{F}_{n}(NOR5)$. This makes the decomposition no longer effective.

\section{Overall results}\label{sec:results}

Our computational study has determined  the value $f(\cC, \lnot \cH)$ for alla cases except  when $\cC$ is the 1-toughness property and $\cH$ is the traceability property. Furthermore, based on the next fact, for each  $n \geq f(\cC, \lnot \cH)$ we are able to construct a 4-regular graph with $n$ vertices that satisfies property $\cC$ and does not satisfy property $\cH$.

\begin{fact} \label{fact:results} Let $G$ be a  $(\cC, \lnot \cH)$-graph with $n$ vertices. Then\\
(i) if $G$ contains a subgraph $H$ isomorphic to  $K_{4}$ then for each $k \in \mathbb{N}$   the graph $G'$ obtained by replacing $H$ by any of the two graphs $T_{4+2k}$ and  $T_{5+2k}$ in Fig. \ref{fig:gadgets1} is a $(\cC,\lnot \cH)$-graph with either $n+ 2k$ or $n+ 1 +2k$ vertices;\\
(ii) if $G$ contains a subgraph $H$ isomorphic to the graph $R_{5}$ then for each   $k \in \mathbb{N}$ the graph $G'$ obtained by replacing $H$ by any of the two graphs $R_{5+2k}$ and $R_{6+2k}$  in Fig. \ref{fig:gadgets2} is a $(\cC,\lnot \cH)$-graph with either $n + 2k$ or $n+ 1 +2k$ vertices.
\end{fact}
\begin{proof} The graphs $T_{4+2k}$ and  $T_{5+2k}$ are 1-tough graphs having four vertices of degree 3 and the other vertices of degree 4.  Thus each of them may be substituted for any complete subgraph of $G$ with 4 vertices leading to a graph $G'$ which still satisfies property $\cC$. The graph $G'$ cannot satisfy  property $\cH$, otherwise, being $H$ a  complete graph, also $G$ would satisfy this property. A similar argument can be used to prove statement (ii).
\end{proof}

\begin{figure}
\centering
\hspace{0cm}
\begin{subfigure}[b]{0.40\linewidth}
\begin{tikzpicture}[scale=0.60, transform shape]
\foreach \x in {   5, 6.5, 8}{
\draw[fill] (\x,0) circle (1.5pt);
\draw[fill] (\x,2.5) circle (1.5pt);
};

\node[shape=circle, radius=2pt] (x1) at (1.5,2.5){} ;
\node[shape=circle, radius=2pt] (x2) at (1.5,0){} ;
\node[shape=circle, radius=2pt] (v1) at (3.5,2.5){} ;
\draw[fill] (3.5,2.5) circle (1.5pt);
\draw[fill] (3.5,0) circle (1.5pt);
\node(e1) at (3.2,3){\Large $v_{1}$} ;
\node[shape=circle, radius=2pt] (v2) at (3.5,0){} ;
\node(e2) at (3.2,-0.5){\Large $v_{2}$} ;
\node(e2) at (8.5,3){\Large $v_{3}$} ;
\node(e4) at (8.5,-0.5){\Large $v_{4}$} ;
\node(e4) at (6.5,3){\Large $v_{3 + 2k}$} ;
\node(e5) at (6.5,-0.5){\Large $v_{4 + 2k}$} ;


\foreach \x in {0, 2.5}{
\draw (3.5,\x) -- (5,0);
\draw (3.5,\x) -- (5,2.5);
\draw[dashed] (6.5,\x) -- (5,0);
\draw[dashed] (6.5,\x) -- (5,2.5);
\draw[dashed] (8,\x) -- (6.5,0);
\draw[dashed] (8,\x) -- (6.5,2.5);};

\draw (3.5, 0) -- (3.5,2.5);
\draw (8, 0) -- (8,2.5);


\end{tikzpicture}
\end{subfigure}
\begin{subfigure}[b]{0.40\linewidth}
\begin{tikzpicture}[scale=0.60, transform shape]
\foreach \x in {   5, 6.5, 8}{
\draw[fill] (\x,0) circle (1.5pt);
\draw[fill] (\x,2.5) circle (1.5pt);
};

\node[shape=circle, radius=2pt] (x1) at (1.5,2.5){} ;
\node[shape=circle, radius=2pt] (x2) at (1.5,0){} ;
\node[shape=circle, radius=2pt] (v1) at (3.5,2.5){} ;
\draw[fill] (3.5,2.5) circle (1.5pt);
\draw[fill] (3.5,0) circle (1.5pt);
\draw[fill] (4.25,1.25) circle (1.5pt);
\draw[fill] (3.5,2.5) circle (1.5pt);
\draw[fill] (3.5,0) circle (1.5pt);

\node(e1) at (3.2,3){\Large $v_{1}$} ;
\node[shape=circle, radius=2pt] (v2) at (3.5,0){} ;
\node(e2) at (3.2,-0.5){\Large $v_{2}$} ;
\node(e2) at (8.5,3){\Large $v_{3}$} ;
\node(e4) at (8.5,-0.5){\Large $v_{4}$} ;
\node(e5) at (4.8, 1.25) {\Large $v_{5}$} ;
\node(e4) at (6.5,3){\Large $v_{3 + 2k}$} ;
\node(e5) at (6.5,-0.5){\Large $v_{4 + 2k}$} ;


\foreach \x in {0, 2.5}{
\draw (3.5,\x) -- (5,0);
\draw (3.5,\x) -- (5,2.5);
\draw[dashed] (6.5,\x) -- (5,0);
\draw[dashed] (6.5,\x) -- (5,2.5);
\draw[dashed] (8,\x) -- (6.5,0);
\draw[dashed] (8,\x) -- (6.5,2.5);};

\draw (3.5, 0) -- (3.5,2.5);
\draw (8, 0) -- (8,2.5);


\end{tikzpicture}
\end{subfigure}

\caption{The graph $T_{4 + 2k}$ with $4 + 2k$ vertices (on the left) and the graph $T_{5 + 2k}$ with $5 + 2k$ vertices (on the right).}
\label{fig:gadgets1}
\end{figure}
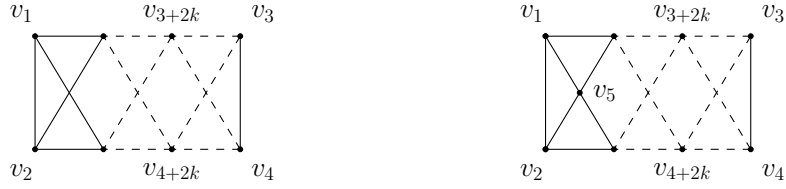

\begin{figure}
\centering
\hspace{0cm}
\begin{subfigure}[b]{0.40\linewidth}
\begin{tikzpicture}[scale=0.60, transform shape]
\foreach \x in {   5, 6.5, 8}{
\draw[fill] (\x,0) circle (1.5pt);
\draw[fill] (\x,2.5) circle (1.5pt);
};
%

\node[shape=circle, radius=2pt] (x1) at (1.5,2.5){} ;
\node[shape=circle, radius=2pt] (x2) at (1.5,0){} ;
\node[shape=circle, radius=2pt] (v1) at (3.5,2.5){} ;
\draw[fill] (3.5,2.5) circle (1.5pt);
\draw[fill] (3.5,0) circle (1.5pt);
\draw[fill] (9.5,1.25) circle (3pt);
\node(e1) at (3.2,3){\Large $v_{1}$} ;
\node[shape=circle, radius=2pt] (v2) at (3.5,0){} ;
\node(e2) at (10,1.25){\Large $v_{3}$} ;
\node(e3) at (3.2,-0.3){\Large $v_{2}$} ;
\node(e4) at (4.8,3){\Large $v_{4}$} ;
\node(e5) at (4.8,-0.5){\Large $v_{5}$} ;
\node(e4) at (7.6,3){\Large $v_{4 + 2k}$} ;
\node(e5) at (7.6,-0.5){\Large $v_{5 + 2k}$} ;


\foreach \x in {0, 2.5}{
\draw (3.5,\x) -- (5,0);
\draw (3.5,\x) -- (5,2.5);
\draw[dashed] (6.5,\x) -- (5,0);
\draw[dashed] (6.5,\x) -- (5,2.5);
\draw[dashed] (8,\x) -- (6.5,0);
\draw[dashed] (8,\x) -- (6.5,2.5);};

\draw (8,0) -- (8,2.5);
\draw (8,0) -- (9.5,1.25);
\draw (8,2.5) -- (9.5,1.25);

\path[-](3.5,0) edge [bend right=80] (9.5,1.25);
\path[-](3.5,2.5) edge [bend left=80] (9.5,1.25);

\end{tikzpicture}
\end{subfigure}
\begin{subfigure}[b]{0.40\linewidth}
\begin{tikzpicture}[scale=0.60, transform shape]
\foreach \x in {   5, 6.5, 8}{
\draw[fill] (\x,0) circle (1.5pt);
\draw[fill] (\x,2.5) circle (1.5pt);
};
%

\node[shape=circle, radius=2pt] (x1) at (1.5,2.5){} ;
\node[shape=circle, radius=2pt] (x2) at (1.5,0){} ;
\node[shape=circle, radius=2pt] (v1) at (3.5,2.5){} ;
\draw[fill] (3.5,2.5) circle (1.5pt);
\draw[fill] (3.5,0) circle (1.5pt);
\draw[fill] (9.5,1.25) circle (1.5pt);
\draw[fill] (4.25,1.25) circle (1.5pt);

\node(e1) at (3.2,3){\Large $v_{1}$} ;
\node[shape=circle, radius=2pt]  at (3.5,0){} ;
\node(e2) at (10,1.25){\Large $v_{3}$} ;
\node(e3) at (3.2,-0.3){\Large $v_{2}$} ;
\node(e4) at (4.8,3){\Large $v_{5}$} ;
\node(e5) at (4.8,-0.5){\Large $v_{6}$} ;
\node(e6) at (3.7,1.25){\Large $v_{4}$} ;
\node(e4k) at (7.6,3){\Large $v_{5 + 2k}$} ;
\node(e5k) at (7.6,-0.5){\Large $v_{6 + 2k}$} ;


\foreach \x in {0, 2.5}{
\draw (3.5,\x) -- (5,0);
\draw (3.5,\x) -- (5,2.5);
\draw[dashed] (6.5,\x) -- (5,0);
\draw[dashed] (6.5,\x) -- (5,2.5);
\draw[dashed] (8,\x) -- (6.5,0);
\draw[dashed] (8,\x) -- (6.5,2.5);};

\draw (8,0) -- (8,2.5);
\draw (8,0) -- (9.5,1.25);
\draw (8,2.5) -- (9.5,1.25);

\path[-](3.5,0) edge [bend right=80] (9.5,1.25);
\path[-](3.5,2.5) edge [bend left=80] (9.5,1.25);

\end{tikzpicture}
\end{subfigure}

\caption{The graphs $R_{5 + 2k}$ with $5 + 2k$ vertices (on the left) and the graph $R_{6 + 2k}$ with $6 + 2k$ vertices (on the right).}
\label{fig:gadgets2}

\end{figure}
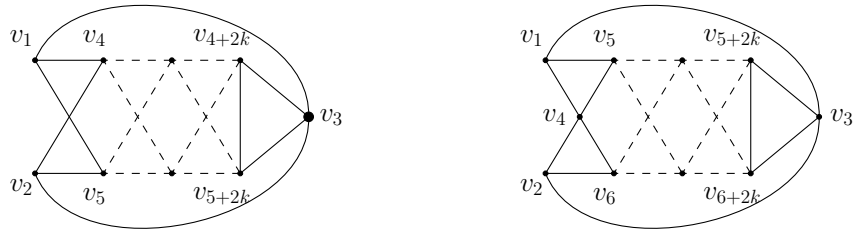

The results that we obtained for the different problems $\genericproblem$ can be summarized as follows:
\subsubsection*{Connectivity} As remarked in Section \ref{sec:notation}, the only open question about the hamiltonian properties of a connected 4-regular graph concerns the traceability. Moreover, by Theorem \ref{thm:CS}~ii) it is $f(c, \lnot T) \geq 16$. 
\begin{prop} \label{prop:connectivity}
Every connected $4$-regular graph with $n\leq 17$ is traceable. Moreover, for every $n \geq 18$ there exists a connected 4-regular graph with $n$ vertices that is not traceable. Thus $f(c, \lnot T) = 18$.
\end{prop}
Indeed, the model $\model{n}{c}{T}$  happened to be infeasible for  $n=16 $ and $n=17$. Moreover, for $n = 18$ our code produced the nontraceable graph  in Fig. \ref{fig:n18bis} which contains a subgraph isomorphic to $R_{5}$. Then the statement follows from Fact \ref{fact:results}.

\begin{figure}
\centering
\begin{minipage}[t]{0.45\linewidth}
\centering
\begin{tikzpicture}[scale=0.5, transform shape]

 spigoli dell'R5 a sinistra 
\draw[xshift=0.0] (1.618,1.17557) \foreach \x in {108,180,252,324} {
-- (\x:2)};
\foreach \x in {252,324,36,108,180}
\draw[fill] (\x:2) circle (2pt);

\draw[xshift=0.0] (-2,0) \foreach \x in {324,108,252,36,180} {
-- (\x:2)};

\draw[xshift=8cm] (-1.618,-1.17557) \foreach \x in {288,360,72,144} {
-- (\x:2)};
[xshitf=8cm]\foreach \x in {72,144,216,288,360}
\draw[fill, xshift=8cm] (\x:2) circle (2pt);

\draw[xshift=8cm] (2,0) \foreach \x in {144,288,72,216,360} {
-- (\x:2)};

\draw[xshift=4cm, yshift=4cm] (1.17557,-1.618) \foreach \x in {-54, 18, 90, 162, 234} {
-- (\x:2)};
[xshitf=4cm, yshift=4cm]\foreach \x in {-54, 18, 90, 162, 234}
\draw[fill, xshift=4cm, yshift=4cm] (\x:2) circle (2pt);

\draw[xshift=4cm, yshift=4cm] (1.17557,-1.618) \foreach \x in {90,234,18,162,-54} {
-- (\x:2)};

\node[ xshift=4cm] (1) at (-1,0) {};
\fill (1) circle[radius=2.5pt];
\node[ xshift=4cm] (2) at (1,0) {};
\fill (2) circle[radius=2.5pt];
\node[ xshift=4cm] (3) at (0,1) {};
\fill (3) circle[radius=2.5pt];
\draw[xshift=4cm] (-1,0) -- (-2.3825, 1.17557);
\draw[xshift=4cm] (-1,0) -- (-2.3825, -1.17557);
\draw[xshift=4cm] (1,0) -- (2.3825, 1.17557);
\draw[xshift=4cm] (1,0) -- (2.3825, -1.17557);
\draw[xshift=4cm] (0,1) -- (-1.17557, 2.382);
\draw[xshift=4cm] (0,1) -- (1.17557, 2.382);
\draw[xshift=4cm] (1,0) -- (0,1);
\draw[xshift=4cm] (1,0) -- (-1,0);
\draw[xshift=4cm] (0,1) -- (-1,0);

\end{tikzpicture}
\caption{A  4-regular connected graph with $n=18$ that is not traceable.} 
\label{fig:n18bis}
\end{minipage}
\hspace{0.5cm}
\begin{minipage}[t]{0.45\linewidth}
\centering
\begin{tikzpicture}[scale=0.4, transform shape]

\draw (1.618,1.17557) \foreach \x in {108,180,252,324} {
-- (\x:2)};
\foreach \x in {252,324,36,108,180}
\draw[fill] (\x:2) circle (2pt);
\draw (-2,0) \foreach \x in {324,108,252,36,180} {
-- (\x:2)};

\draw[xshift=4cm] (1.618,1.17557) \foreach \x in {108,180,252,324} {
-- (\x:2)};
\foreach \x in {252,324,36,108,180}
\draw[fill, xshift=4cm] (\x:2) circle (2pt);
\draw[xshift=4cm] (-2,0) \foreach \x in {324,108,252,36,180} {
-- (\x:2)};

\draw[xshift=8cm] (1.618,1.17557) \foreach \x in {108,180,252,324} {
-- (\x:2)};
\foreach \x in {252,324,36,108,180}
\draw[fill, xshift=8cm] (\x:2) circle (2pt);
\draw[xshift=8cm] (-2,0) \foreach \x in {324,108,252,36,180} {
-- (\x:2)};

\draw[xshift=12cm] (1.618,1.17557) \foreach \x in {108,180,252,324} {
-- (\x:2)};
\foreach \x in {252,324,36,108,180}
\draw[fill, xshift=12cm] (\x:2) circle (2pt);
\draw[xshift=12cm] (-2,0) \foreach \x in {324,108,252,36,180} {
-- (\x:2)};

\draw[fill] (6,5) circle (2pt);
\draw[fill] (6,-5) circle (2pt);

\draw (6,5) -- (36:2);
\draw[xshift=4cm] (6-4,5) -- (36:2);
\draw[xshift=8cm] (6-8,5) -- (36:2);
\draw[xshift=12cm] (6-12,5) -- (36:2);

\draw (6,-5) -- (324:2);
\draw[xshift=4cm] (6-4,-5) -- (324:2);
\draw[xshift=8cm] (6-8,-5) -- (324:2);
\draw[xshift=12cm] (6-12,-5) -- (324:2);

\end{tikzpicture}
\caption{4-regular $2$-connected graph with $n=22$ that is not traceable.}
\label{Figure:n22}

\end{minipage}
\end{figure}

\subsubsection*{2-connectivity} By  Theorem \ref{thm:hilbig}, every 2-connected 4-regular graph $G$ with less than 16 vertices is hamiltonian, thus homogeneously  traceable and traceable. 
\begin{prop} \label{prop:2connectivityHHT} For every $n \geq 16$ there exists a 2-connected 4-regular graph with $n$ vertices which is not 1-tough, thus it is neither hamiltonian nor homogeneously traceable. This implies that $f(2c, \lnot H) = f(2c, \lnot HT) = 16$.
\end{prop} 
Indeed, the preliminary analysis found out the graph in Fig. \ref{Figure:n16} which is 2-connected but not 1-tough. Since this graph contains an $R_{5}$, the statement follows from Fact \ref{fact:results}.

\begin{prop} \label{prop:2connectivityT} Every 2-connected $4$-regular graph with $n\leq 21$ is traceable. Moreover for every $n \geq 22$ there exists a 2-connected  $4$-regular graph with 22 vertices which is nontraceable. Thus $f(2c, \lnot T) = 22$.
\end{prop}
The result follows from Proposition \ref{prop:connectivity} when $n \leq 17$. Moreover  the model  $\model{n}{2c}{T}$  happened to be infeasible for every $18 \leq n \leq 21$.  Finally, for $n=22$ our code produced the 2-connected nontraceable  graph in Fig. \ref{Figure:n22}. Since this graph contains an $R_{5}$, the statement follows from Fact \ref{fact:results}.

\subsubsection*{1-toughness} 

\begin{prop} \label{prop:1toughH} Every 1-tough $4$-regular graph with $n \leq 17$ is hamiltonian. Moreover for every $n \geq 18$ there exists a 1-tough $4$-regular graph with $n$ vertices which is not hamiltonian. Thus $f(1t, \lnot H) = 18$.
\end{prop} 
Indeed, the model $\model{n}{1t}{H}$  happened to be infeasible for  $n=16 $ and $n=17$. Moreover, for $n = 18$ our code produced the same 1-tough nonhamiltonian graph  in Fig.\ref{Figure:n18} proposed in  \cite{BBV90}. Since this graph contains a $K_{4}$, the statement follows from Fact \ref{fact:results}.

\begin{prop} \label{prop:1-toughHT} Every 1-tough $4$-regular graph with $n \leq 19$ is homogeneously traceable. Moreover for every $n \geq 20$ there exists a 1-tough  $4$-regular graph with $n$ vertices which is not homogeneously traceable. Thus $f(1t, \lnot HT) = 20$.
\end{prop}

The  result follows from Proposition \ref{prop:1toughH} when $n \leq 17$. Moreover  the model  $\model{n}{1t}{HT}$  happened to be infeasible for $n=18$ and $n=19$.  Finally, for $n=20$ our code produced the 1-tough 4-regular graph in Fig. \ref{Figure:n20} which is not homogeneously traceable since it does not contain any Hamilton path from vertex $v$. Since this graph contains both a $K_{4}$ and an $R_{5}$ the statement follows from Fact \ref{fact:results}.

\begin{prop} \label{prop:1-toughT} Every 1-tough $4$-regular graph with $n \leq 21$ is traceable. Moreover for every $n \geq 40$ there exists a 1-tough  $4$-regular nontraceable graph with $n$ vertices. Thus $22 \leq f(1t, \lnot T) \leq 40$.
\end{prop}

The bound  $f(1t, T) \geq 22$ follows from Proposition \ref{prop:2connectivityT}. Moreover, the graph with 40 vertices in Fig. \ref{Figure:n40} that is obtained by suitably connecting two copies of the graph in Fig. \ref{Figure:n20} happens to be a 1-tough 4-regular graph which is not traceable. Since this graph contains both a $K_{4}$ and an $R_{5}$ the statement follows from Fact \ref{fact:results}. 

The above results allow us to fill the entries in boldface of Table \ref{tab:result1} in Section \ref{sec:introduction}.

\begin{figure}[t]
\centering
\begin{minipage}[c]{0.45\linewidth}
\centering
\vspace{1cm}
\begin{tikzpicture}[scale=0.4, transform shape]

\draw[xshift=-4cm] (1.618,1.17557) \foreach \x in {108,180,252,324} {
-- (\x:2)};
\foreach \x in {252,324,36,108,180}
\draw[fill, xshift=-4cm] (\x:2) circle (2pt);
\draw[xshift=-4cm] (-2,0) \foreach \x in {324,108,252,36,180} {
-- (\x:2)};

\draw[xshift=4cm] (-1.618,-1.17557) \foreach \x in {288,360,72,144} {
-- (\x:2)};
[xshitf=8cm]\foreach \x in {72,144,216,288,360}
\draw[fill, xshift=4cm] (\x:2) circle (2pt);
\draw[xshift=4cm] (2,0) \foreach \x in {144,288,72,216,360} {
-- (\x:2)};

\foreach \x in {-2,-1,1,2}{
\draw[fill] (1,\x) circle (2pt);
\draw[fill] (-1,\x) circle (2pt);
\draw (-1,\x) -- (1,\x);
}
\foreach \x in{-1,1}{
\draw (\x,-2) -- (\x,-1);
\draw (\x,1) -- (\x,2);
}
\foreach \x in{-2,1}{
\draw (1,\x) -- (-1,\x+1);
\draw (1,\x+1) -- (-1,\x);
}
\draw[fill] (0,0) circle (2pt);
\node[shape=star , star points=4, inner sep=0pt, draw, minimum size=6pt, thick] (v1) at (0,3) {};
\node[shape=circle] (2) at (0,3.4) {\LARGE$v$} ;

\draw (0,0) -- (-1,-1);
\draw (0,0) -- (-1,1);
\draw (0,0) -- (1,-1);
\draw (0,0) -- (1,1);

\draw (0,3) -- (-1,2);
\draw (0,3) -- (1,2);
\draw[xshift=-4cm] (4,3) -- (36:2);
\draw[xshift=4cm] (-4,3) -- (144:2);

\draw[xshift=-4cm] (3,-2) -- (324:2);
\draw[xshift=4cm] (-3,-2) -- (216:2);

\end{tikzpicture}
\vspace{1.1cm}
\caption{A 4-regular $1$-tough graph with $n=20$ that is  not homogeneously traceable. There is not a Hamilton path starting at $v$.}
\label{Figure:n20}
\end{minipage}
\hspace{0.5cm}
\begin{minipage}[c]{0.45\linewidth}
\centering
\begin{tikzpicture}[scale=0.4, transform shape]

\draw[xshift=-4cm] (1.618,1.17557) \foreach \x in {108,180,252,324} {
-- (\x:2)};
\foreach \x in {252,324,36,108,180}
\draw[fill, xshift=-4cm] (\x:2) circle (2pt);
\draw[xshift=-4cm] (-2,0) \foreach \x in {324,108,252,36,180} {
-- (\x:2)};

\draw[xshift=4cm] (-1.618,-1.17557) \foreach \x in {288,360,72,144} {
-- (\x:2)};
\foreach \x in {72,144,216,288,360}
\draw[fill, xshift=4cm] (\x:2) circle (2pt);
\draw[xshift=4cm] (2,0) \foreach \x in {144,288,72,216,360} {
-- (\x:2)};

\draw[xshift=-4cm, yshift=6cm] (1.618,1.17557) \foreach \x in {108,180,252,324} {
-- (\x:2)};
\foreach \x in {252,324,36,108,180}
\draw[fill, xshift=-4cm, yshift=6cm] (\x:2) circle (2pt);
\draw[xshift=-4cm, yshift=6cm] (-2,0) \foreach \x in {324,108,252,36,180} {
-- (\x:2)};

\draw[xshift=4cm, yshift=6cm] (-1.618,-1.17557) \foreach \x in {288,360,72,144} {
-- (\x:2)};
\foreach \x in {72,144,216,288,360}
\draw[fill, xshift=4cm, yshift=6cm] (\x:2) circle (2pt);
\draw[xshift=4cm, yshift=6cm] (2,0) \foreach \x in {144,288,72,216,360} {
-- (\x:2)};

\foreach \x in {-2,-1,1,2,4,5,7,8}{
\draw[fill] (1,\x) circle (2pt);
\draw[fill] (-1,\x) circle (2pt);
\draw (-1,\x) -- (1,\x);
}
\foreach \x in{-1,1}{
\draw (\x,-2) -- (\x,-1);
\draw (\x,1) -- (\x,2);
\draw (\x,4) -- (\x,5);
\draw (\x,7) -- (\x,8);
}
\foreach \x in{-2,1,4,7}{
\draw (1,\x) -- (-1,\x+1);
\draw (1,\x+1) -- (-1,\x);
}
\draw[fill] (0,0) circle (2pt);
\draw[fill] (0,6) circle (2pt);
\draw[fill] (-2,3) circle (2pt);
\draw[fill] (2,3) circle (2pt);
\draw (0,0) -- (-1,-1);
\draw (0,0) -- (-1,1);
\draw (0,0) -- (1,-1);
\draw (0,0) -- (1,1);

\draw (0,6) -- (-1,5);
\draw (0,6) -- (1,5);
\draw (0,6) -- (-1,7);
\draw (0,6) -- (1,7);

\draw (2,3) -- (-1,4);
\draw (2,3) -- (1,4);
\draw[xshift=4cm] (2-4,3) -- (144:2);
\draw[xshift=4cm,yshift=6cm] (2-4,3-6) -- (216:2);

\draw (-2,3) -- (-1,2);
\draw (-2,3) -- (1,2);
\draw[xshift=-4cm] (4-2,3) -- (36:2);
\draw[xshift=-4cm,yshift=6cm] (-2+4,3-6) -- (324:2);

\draw[xshift=-4cm] (-1+4,-2) -- (324:2);
\draw[xshift=4cm] (1-4,-2) -- (216:2);

\draw[xshift=-4cm, yshift=6cm] (-1+4,8-6) -- (36:2);
\draw[xshift=4cm, yshift=6cm] (1-4,8-6) -- (144:2);

\end{tikzpicture}
\caption{A 4-regular $1$-tough graph with $n=40$ that is not traceable.}
\label{Figure:n40}
\end{minipage}
\end{figure}

\section{Computational results} \label{sec:computationalresults}

All the models $\genericmodel$  were solved within the SCIP 6.0.0 framework for branch and cut  \cite{scip}, using CPLEX 12.4 \cite{cplex2009v12} as the  LP solver and the software Nauty (\cite{nauty}) to perform orbital branching. 
All the experiments were run on an Intel i7 CPU with 3.6GHz, 6 cores (4+2) and 16 GB RAM. 

Next we list some details relative to the solution  of model $\genericmodel$ for the different pairs of properties $(\cC, \lnot \cH)$: 
\begin{itemize}
\item $(c, \lnot T)$: for $16 \leq n \leq 18$ we solved model $\model{n}{c}{T}$  by solving the subproblems $Q(n, c, \lnot P_{k})$ ,  $k=8, \dots, n-1$ (see Subsection \ref{subsec:decomposition}). The separation of the connectivity constraints was performed on the integer solutions by checking the connectivity of the corresponding graph. The solution in Fig. \ref{fig:n18bis} was returned when solving problem $Q(18, c, \lnot P_{13})$;
\item $(2c, \lnot T)$: for $16 \leq n \leq 22$ we solved model $\model{n}{2c}{T}$  by solving the subproblems $Q(n, 2c, \lnot P_{k})$, $k=8, \dots, n-1$. The separation of the 2-connectivity constraints was performed on the integer solutions by checking the 2-connectivity of the corresponding graph. The solution in Fig. \ref{Figure:n22} was returned when solving problem $Q(22, 2c, \lnot P_{17})$;
\item $(1t, \lnot H)$: for $16 \leq n \leq 18$ the feasible set of model $\model{n}{1t}{H}$ was partitioned in the three sets ${\cal F}(NOR5)$,  ${\cal F}(1R5)$ and  ${\cal F}(2R5)$ containing the feasible solutions with, respectively, no $R_{5}$, a single $R_{5}$ and at least two disjoint $R_{5}$. The separation of the 1-tough constraints was performed as explained in Subsection \ref{subsec:decomposition} and the models with feasible set  ${\cal F}(NOR5)$ were solved by solving the subproblems  $Q(n, 1t, \lnot P_{k})$ for $k=8, \dots, n$. The solution in Fig.\ref{Figure:n18} was returned when solving the subproblem $Q(18, 1t, \lnot H)$;
\item $(1t, \lnot HT)$: for $18 \leq n \leq 20$ the model $\model{n}{1t}{HT}$ was solved by solving the subproblems $Q(n, 2c, \lnot P_{k}(1))$, $k=8, \dots, n-1$. The 1-toughness constraints have been separated by solving the problems $T_{\ell,\ell+1}$ with $\ell \in \{2,4\}$ for $n=18$ and $\ell \in \{2,4,5\}$ for $n=19,20$. The solution in Fig.\ref{Figure:n20} was returned when solving the subproblem $Q(20, 1t, \lnot P_{15}(1))$. Since there exists a single 2-connected graph which disconnects in 4 components by removing three vertices and this graph contains a path of length at least 17 from any vertex, the problem $T_{3,4}$ was not solved for  $k \leq 15$.  
\end{itemize}

The computational times required to solve the different models $\genericmodel$ are reported in Table \ref{Table:computation}. For each pair of properties and each value $n$, the reported time  is the sum of the times needed to solve all the subproblems used for that case. As expected, the computational times significantly increase with the number $n$ of vertices of the instances and with the type of constraints that define model $\genericmodel$. 

\begin{table}[t]
    \centering
    \begin{tabular}{|c|c|r|l|}
        \hline
        
   n & result & time (sec) & time (hours)\\
   \hhline{|=|=|=|=|}
   
       \multicolumn{4}{|l|}{connected and nontraceable} \\
   \hline
   16   & infeasible & $572 $ & $<$ 1 \\
   17   & infeasible & $3344$ & $<$  1\\
   18   & Figure \ref{fig:n18bis} & $510   $ & $<$  1\\
     \hline
      \multicolumn{4}{|l|}{$2$-connected and nontraceable} \\
   \hline
   18   & infeasible & $12510$ & $<$ 4 \\
   19   & infeasible & $66023$ & $<$ 19 \\
   20   & infeasible & $221054$ & $<$ 62  \\
   21   & infeasible & $1558501$ & $<$ 433 \\
   22   & Figure \ref{Figure:n22} & 5128 & $<$ 2\\
   \hline
     \multicolumn{4}{|l|}{$1$-tough and nonhamiltonian} \\
   \hline
   16   & infeasible & $3978  $ & $<$ 2\\
   17   & infeasible & $38856 $ & $<$ 11\\
   18   & Figure \ref{Figure:n18} & $24620 $ &$<$ 7\\
   \hline
         \multicolumn{4}{|l|}{$1$-tough  and non-homogeneously traceable} \\
   \hline
    18   & infeasible & $57356 $& $<$ 15\\
   19   & infeasible &  $249100 $ & $<$ 70\\
   20   & Figure \ref{Figure:n20} & $24700$ & $<$ 7\\
   \hline

\end{tabular}
    \caption{Computational results. The times are those reported by SCIP at the end of the computation.} \label{Table:computation}

\end{table}



\section{Conclusions}\label{sec:conclusions}

It is known that the hamiltonian properties of a 4-regular graph $G$ depend on its connectivity properties and on its order, but for several pairs ${\cal H}, {\cal C}$ of such properties, determining the smallest order $n$ such that $G$ has ${\cal C}$ but does not have ${\cal H}$ is a challenging problem. In this paper we attacked this problem by using Integer Linear Programming to formulate the search of this type of graphs.
We believe that using ILP to construct combinatorial objects (such as a graph) with given properties is a viable technique,  which we want to 
support with our work. This technique can be used, for example, to try and settle a conjecture on the existence of an object of a particular type, {\em provided its size is ``small enough''}. Indeed, a major limitation of the ILP approach is that the running times grow exponentially with the object's size (as we experienced in our computations) and even for small instances, a certain amount of ingenuity and software engineering is required to make the approach work.

In our paper we were able to settle some open questions about the 
hamiltonian properties of 4-regular graphs of small order, but still large enough to make it  quite hard to address them by a theoretical analysis. Indeed,
even for these small graphs, we were able to complete the
computations in a relatively small amount of time only thanks to 
the adoption of some strategic choices such as (i) the use of
symmetry-breaking tools; (ii) a decomposition of the cases
based on the existence/absence of paths of a given length; (iii) a preliminary analysis aimed at identifying subgraphs of the sought 
 graph with a  well defined structure, whose knowledge allowed us to considerably limit  the search space. We remark 
that even the preliminary analysis was carried out
by using an ILP approach, which demonstrates not only the  the power, 
but also the flexibility
of using Integer Programming when searching for combinatorial 
objects with given properties. 

Although we were able to determine the smallest order for 
$({\cal C},\lnot{\cal H})$-graphs for almost all pairs $({\cal C},{\cal H})$, one remaining open problem concerns the nontraceability of 1-tough 
graphs. For this problem, the order of a smallest graph would
still be too large for solving the corresponding ILP model within
an acceptable time (it is difficult to estimate how much it would take,
but, based on our experience, months or even years).

A future direction for our research would be to identify some other open problems
or conjectures in graph theory concerning graphs of relatively small order, and try to tackle these problems with an ILP approach. We believe the 
use of this tool can be vary helpful in closing some of these questions.


\bibliography{Bib_article}

\begin{thebibliography}{10}

\bibitem{BBS06}
{\sc Bauer, D., Broersma, H., and Schmeichel, E.}
\newblock Toughness in graphs - a survey.
\newblock {\em Graphs and Combinatorics 22}, 1 (2006), 1--35.

\bibitem{BBV90}
{\sc Bauer, D., Broersma, H.~J., and Veldman, H.~J.}
\newblock On smallest nonhamiltonian regular tough graphs.
\newblock {\em Congressus numerantium 70\/} (1990), 95--98.

\bibitem{BHS90}
{\sc Bauer, D., Hakimi, S.~L., and Schmeichel, E.}
\newblock Recognizing tough graphs is {NP}-hard.
\newblock {\em Discrete Applied Mathematics 28}, 3 (1990), 191--195.

\bibitem{BHMS97}
{\sc Bauer, D., van~den Heuvel, J., Morgana, A., and Schmeichel, E.}
\newblock The complexity of recognizing tough cubic graphs.
\newblock {\em Discrete Applied Mathematics 79}, 1--3 (1997), 35--44.

\bibitem{B79}
{\sc Bermond, J.-C.}
\newblock Hamiltonian graphs.
\newblock In {\em Selected topics in graph theory}. Academic Press, 1979,
  ch.~6, pp.~127--167.

\bibitem{B78}
{\sc Bondy, J.~A.}
\newblock {\em Hamilton cycles in graphs and digraphs}.
\newblock Department of Combinatorics and Optimization, University of Waterloo,
  1978.

\bibitem{Cap15}
{\sc Caprara, A., Dell’Amico, M., Diaz-Diaz, J., Iori, M., and Rizzi, R.}
\newblock Friendly bin packing instances without integer round-up property.
\newblock {\em Mathematical Programming, Ser. B 150\/} (2015), 5--17.

\bibitem{C73}
{\sc Chvátal, V.}
\newblock Tough graphs and hamiltonian circuits.
\newblock {\em Discrete Mathematics 5}, 3 (1973), 215--228.

\bibitem{CS13}
{\sc Cranston, D.~W., and Suil, O.}
\newblock Hamiltonicity in connected regular graphs.
\newblock {\em Information Processing Letters 113}, 22 (2013), 858--860.

\bibitem{GJ79}
{\sc Garey, M., and Johnson, D.}
\newblock {\em Computers and Intractability: a Guide to the Theory of
  NP-completeness}.
\newblock W.H. Freeman \& Co, 1979.

\bibitem{scip}
{\sc Gleixner, A., and al.}
\newblock {The SCIP Optimization Suite 6.0}.
\newblock Tech. rep., ZIB-Report 18-26, Zuse Institute Berlin, 2018.

\bibitem{G91}
{\sc Gould, R.~J.}
\newblock Updating the hamiltonian problem - a survey.
\newblock {\em Journal of Graph Theory 15}, 2 (1991), 121--157.

\bibitem{H86}
{\sc Hilbig, F.}
\newblock {\em Kantenstrukturen in nichthamiltonschen Graphen. (Edge structures
  in nonhamiltonian graphs)}.
\newblock PhD thesis, Technische Univers{\"i}t at Berlin, 1986.

\bibitem{cplex2009v12}
{\sc {IBM ILOG CPLEX}}.
\newblock V12. 1: User's manual for cplex.
\newblock {\em International Business Machines Corporation 46}, 53 (2009), 157.

\bibitem{J80}
{\sc Jackson, B.}
\newblock Hamilton cycles in regular 2-connected graphs.
\newblock {\em Journal of Combinatorial Theory, Series B 29}, 1 (1980), 27 --
  46.

\bibitem{LPR20}
{\sc Lancia, G., Pippia, E., and Rinaldi, F.}
\newblock Using integer programming to search for counterexamples: A case
  study.
\newblock In {\em Mathematical Optimization Theory and Operations Research
  2020}, A.~Kononov, M.~Khachay, V.~A. Kalyagin, and P.~Pardalos, Eds.,
  vol.~12095 of {\em LNCS}. Springer, 2020, pp.~69--84.

\bibitem{M10}
{\sc Margot, F.}
\newblock Symmetry in integer linear programming.
\newblock In {\em 50 Years of Integer Programming 1958-2008}. Springer, 2010,
  pp.~647--686.

\bibitem{nauty}
{\sc McKay, B.~D., and A., P.}
\newblock Practical graph isomorphism, \{II\}.
\newblock {\em Journal of Symbolic Computation 60}, 0 (2014), 94 -- 112.

\bibitem{O63}
{\sc Ore, O.}
\newblock Hamilton connected graphs.
\newblock {\em Journal des Mathematiques Pures et Appliquees 42\/} (1963),
  21--27.

\bibitem{OLRS11}
{\sc Ostrowski, J., Linderoth, J., Rossi, F., and Smriglio, S.}
\newblock Orbital branching.
\newblock {\em Mathematical Programming 126}, 1 (2011), 147--178.

\bibitem{PR19}
{\sc Pfetsch, M.~E., and Rehn, T.}
\newblock A computational comparison of symmetry handling methods for mixed
  integer programs.
\newblock {\em Mathematical Programming Computation 11}, 1 (2019), 37--93.

\bibitem{P94}
{\sc Picouleau, C.}
\newblock Complexity of the hamiltonian cycle in regular graph problem.
\newblock {\em Theoretical Computer Science 131\/} (1994).

\bibitem{P17}
{\sc Pulaj, J.}
\newblock Cutting planes for families implying frankl's conjecture.
\newblock {\em Mathematics of Computation 89\/} (2020), 829--857.

\bibitem{PRT20}
{\sc Pulaj, J., A., R., and D., T.}
\newblock New conjectures for union-closed families.
\newblock {\em The electronic journal of combinatorics 23\/} (2016).

\bibitem{S98}
{\sc Skiena, S.~S.}
\newblock {\em The algorithm design manual: Text}, vol.~1.
\newblock Springer Science \& Business Media, 1998.

\bibitem{S84}
{\sc Skupie\'n, Z.}
\newblock Homogeneously traceable and hamiltonian connected graphs.
\newblock {\em Demonstratio Mathematica 17}, 4 (1984), 1051--1068.

\bibitem{TSSW96}
{\sc Trevisan, L., Sorkin, G.~B., Sudan, M., and Williamson, D.~P.}
\newblock Gadgets, approximation, and linear programming.
\newblock {\em SIAM Journal on Computing 29}, 6 (2000), 2074--2097.

\end{thebibliography}

\end{document}